\documentclass[11pt]{article}
\usepackage[margin =1in]{geometry}
\usepackage{amssymb,amsthm,amsmath}
\usepackage{enumerate}
\usepackage{hyperref}
\usepackage{cite}
\usepackage[capitalize]{cleveref}
\usepackage{enumitem}
\usepackage{color}
\usepackage{cancel}
\usepackage{pgf}
\usepackage{svg}
\usepackage{pgfplots}
\usepackage{import}

\usepackage[utf8]{inputenc} 
\usepackage[T1]{fontenc}    
\usepackage{hyperref}       
\usepackage{url}            
\usepackage{booktabs}       
\usepackage{amsfonts}       
\usepackage{nicefrac}       
\usepackage{microtype}
\usepackage{multirow}
\usepackage{xfrac}
\usepackage{todonotes}
\usepackage{titlesec}
\usepackage{caption}

\usepackage{multirow}

\usepackage{caption}
\usepackage{float}

\titleformat{\subsubsection}[runin]
{\normalfont\normalsize\bfseries}{\thesubsubsection}{1em}{}

\usepackage{graphicx}

\usepackage{appendix}
\numberwithin{equation}{section}

\newcommand{\inclu}[0] {\ar@{^{(}->}}

\newtheorem{thm}{Theorem}[section]
\newtheorem{theorem}{Theorem}[section]

\newtheorem{conjecture}[thm]{Conjecture}

\crefname{claim}{claim}{claims}
\Crefname{claim}{Claim}{Claims}
\crefname{lem}{lemma}{lemmas}
\Crefname{lem}{Lemma}{Lemmas}
\crefname{algorithm}{algorithm}{algorithms}
\Crefname{algorithm}{Algorithm}{Algorithms}

\theoremstyle{definition}

\theoremstyle{definition}

\theoremstyle{definition}
\newtheorem*{claim*}{Claim}

\usepackage{mathtools}

\usepackage[ruled]{algorithm2e}

\usepackage[T1]{fontenc}
\usepackage{lmodern}
\crefname{figure}{Figure}{Figures}

\usepackage{xcolor}
\definecolor{blue}{RGB}{0,0,0} 

\usepackage{subcaption}

\pgfplotsset{compat=1.18}

\begin{document}

    \title{Some New Insights from Highly Optimized Polyhedral Passages}

    \author{Raj Gosain \qquad Benjamin Grimmer\footnote{Johns Hopkins University, Department of Applied Mathematics and Statistics, \url{grimmer@jhu.edu}}}
    \date{}
    \maketitle
    
    \begin{abstract}
        A shape possesses Rupert's property if a hole can be cut through it such that a second identical copy of the shape can cleanly pass straight through the interior of the first. Such a passage proving cubes are Rupert was first shown more than 300 years ago. It remains open whether every polyhedron in three dimensions is Rupert. We propose a customized subgradient method providing high-accuracy local numerical optimization of the quality of a passage for a given polyhedron. From extensive numerical searches, we improve these best-known passages for more than half of the Platonic, Archimedean, and Catalan solids and for numerous Johnson solids. Our high accuracy solves support a new conjecture of a simple form for the Tetrahedron's optimal passage. Despite our computational search, three Archimedean and two Catalan solids remain open, providing further negative evidence against the conjecture that all polyhedrons are Rupert.
    \end{abstract}

    \section{Introduction.}
Given two equally large cubes, one can pass the first through a hole strictly inside the second. This surprising fact is attributed to J. Wallis and Prince Rupert of the Rhine in the 17th century. In the 18th century, Peter Nieuwland improved Wallis's construction, proving that such a passage can even be made by a cube up to $\mu_\mathtt{cube} = \frac{3\sqrt{2}}{4}$ times larger than the other. This passage, shown in Figure~\ref{fig:platonic} (a), allows for the largest possible rescaling; see~\cite{bezdek2021} for a proof. This largest possible rescaling $\mu$ is known as the Nieuwland constant.

This is not a property unique to cubes. We say a compact convex set $S\subseteq \mathbb{R}^3$ is {\it Rupert} if such a passage exists with $\mu_S>1$; see Section~\ref{sec:methods} for a formal definition. The unit ball is not Rupert (i.e., for the ball, $\mu=1$), so this property does not hold universally. Attention has since shifted to identifying which polyhedra have this property.

Over the past 60 years, several works have considered classic, highly symmetric polyhedra beyond cubes. {\color{blue} Such polyhedra represent a fruitful middle ground between the cube and the ball in which to investigate Rupertness.} First, in 1968, Scriba~\cite{scriba1968} proved that the octahedron and tetrahedron are both Rupert. The remaining two Platonic solids, the dodecahedron and the icosahedron, were proven Rupert in 2007 by Jerrad et al.~\cite{jerrad2007}. In each case, this was done by demonstrating lower bounds on their Nieuwland constant larger than one. From this discovery, Jerrad et al.~made the following optimistic conjecture:
\begin{conjecture}[Jerrad et al.~\cite{jerrad2007}]\label{conj:EveryPolyhedron}
    Every convex polyhedron in $\mathbb{R}^3$ is Rupert.
\end{conjecture}
{\color{blue} After resolving the Platonic solids, interest turned to the family of Archimedean solids.} Passages for 8 out of 13 Archimedean solids were identified by Chai et al.~\cite{Chai2018}. Hoffmann~\cite{Hoffmann19} and Lavau~\cite{Lavau19} subsequently showed that the truncated tetrahedron is Rupert. By characterizing Rupert's property as the existence of a solution to a certain (large) system of polynomial inequalities and applying a randomized computer search, one additional Archimedean solid, the truncated icosidodecahedron, was proved to be Rupert by Steininger and Yurkevich~\cite{steininger2021}. Further, their technique resolved 9 of the 13 Catalan solids, {\color{blue} the family of face-transitive shapes dual to the Archimedean solids}, and 82 of the 92 Johnson solids, {\color{blue} the family of all convex polyhedra with regular polygons as faces~\cite{Johnson1966}}. Furthering this direction of computer-aided search, Fredriksson~\cite{fredriksson2024} converted computing the Nieuwland constant into a certain four-dimensional nonsmooth optimization problem, seeking a passage facilitating the maximum possible rescaling. Locally, numerically solving this optimization problem found improved lower bounds on the Nieuwland constant for several Archimedean, Catalan, and Johnson solids. Although no progress was made on identifying new Rupert Archimedean solids therein, two new Catalan solids and five new Johnson solids were shown to be Rupert.

\begin{figure}
    \centering
    \begin{subfigure}[b]{0.3\textwidth}
        \centering
        \includegraphics[width=\linewidth, trim={2.1cm 2.1cm 2.1cm 2.1cm}, clip]{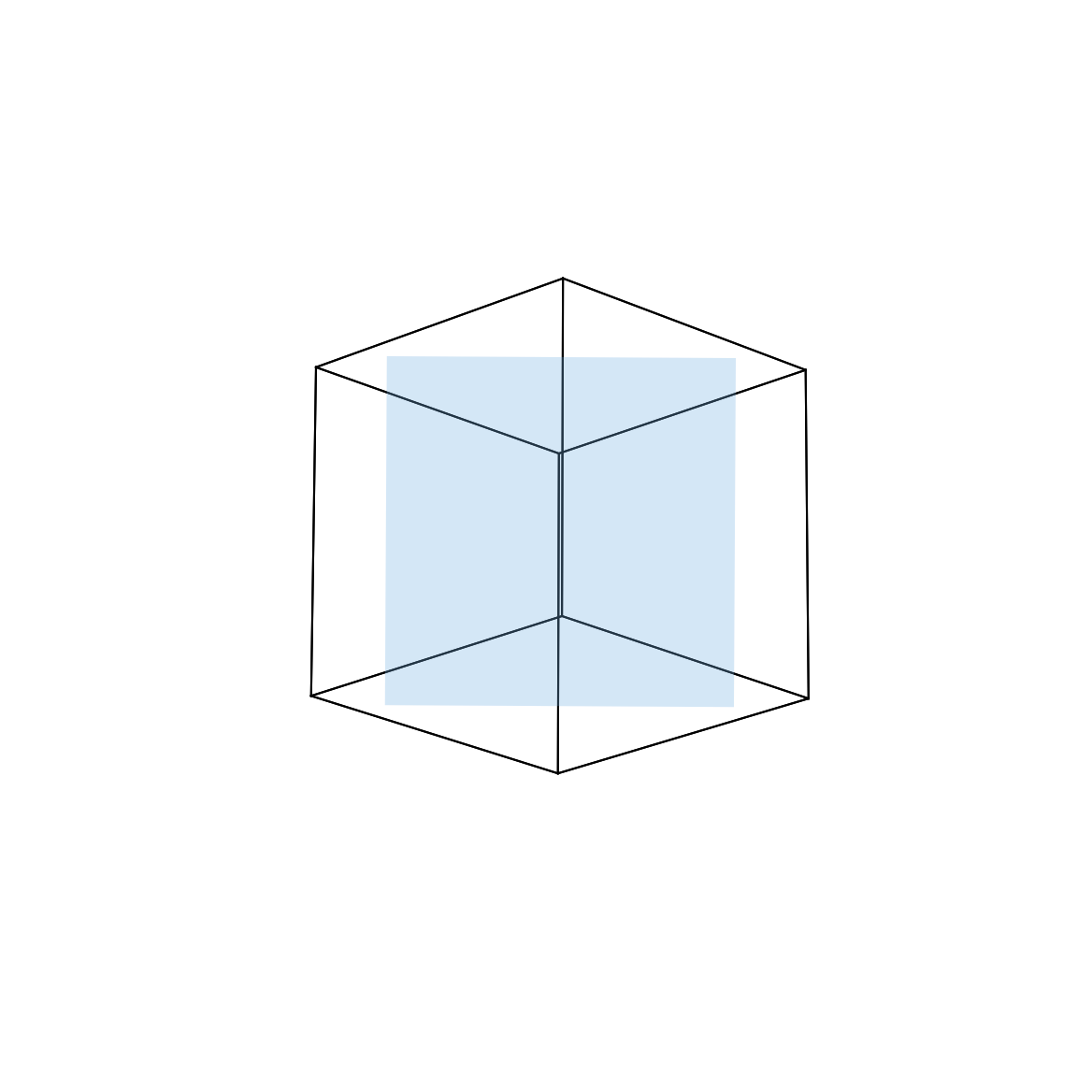}
        \caption{\centering Cube\newline $\mu=\frac{3\sqrt{2}}{4}$}
        \label{fig:cube}
    \end{subfigure}
    \hfill
    \begin{subfigure}[b]{0.3\textwidth}
        \centering
        \includegraphics[width=\linewidth, trim={2.1cm 2.1cm 2.1cm 2.1cm}, clip]{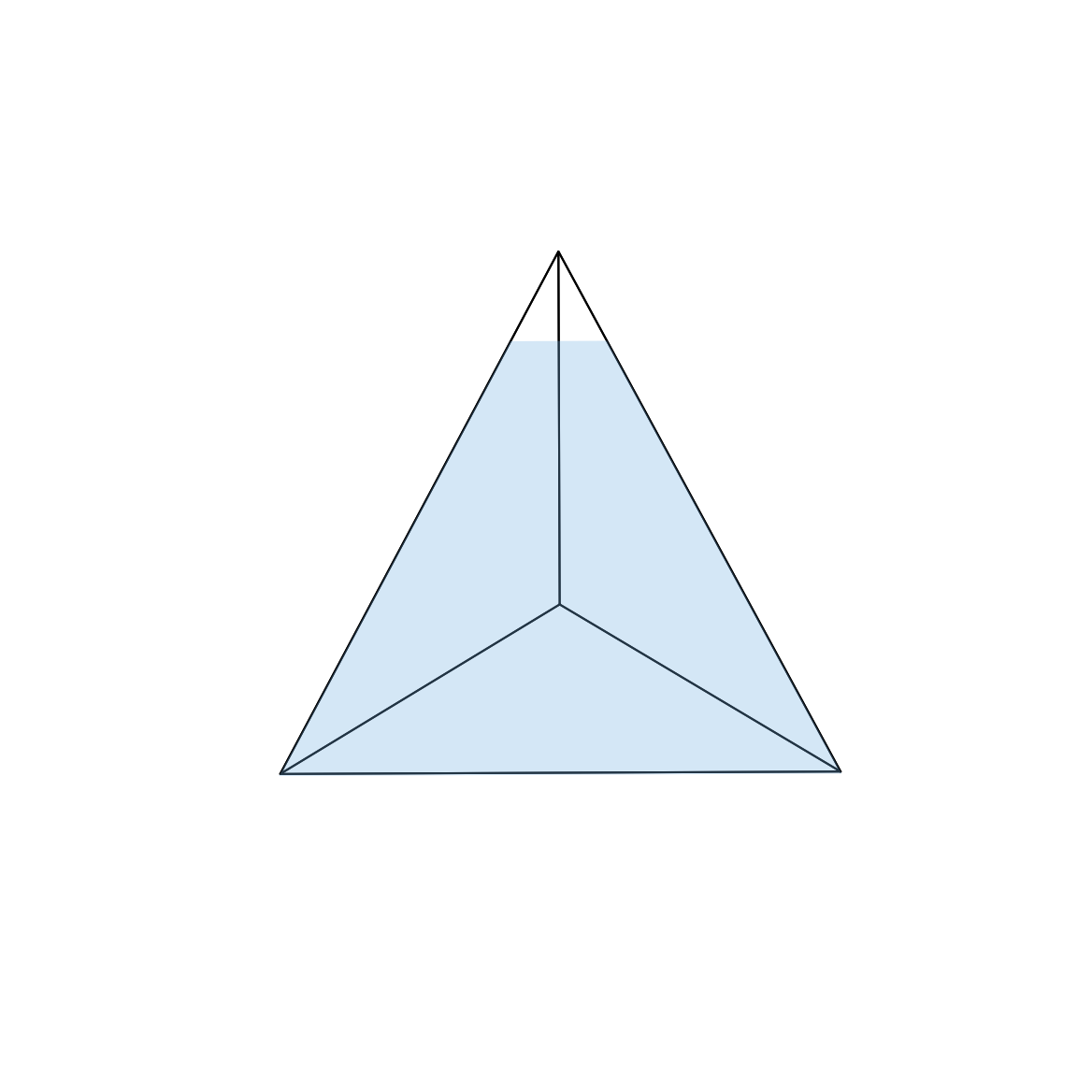}
        \caption{\centering Tetrahedron\newline $\mu \geq \mathbf{\frac{\sqrt{6}}{1+\sqrt{2}}}>\mathbf{1.014611872334}$}
        \label{fig:tetrahedron}
    \end{subfigure}
    \hfill
    \begin{subfigure}[b]{0.3\textwidth}
        \centering
        \includegraphics[width=\linewidth, trim={2.1cm 2.1cm 2.1cm 2.1cm}, clip]{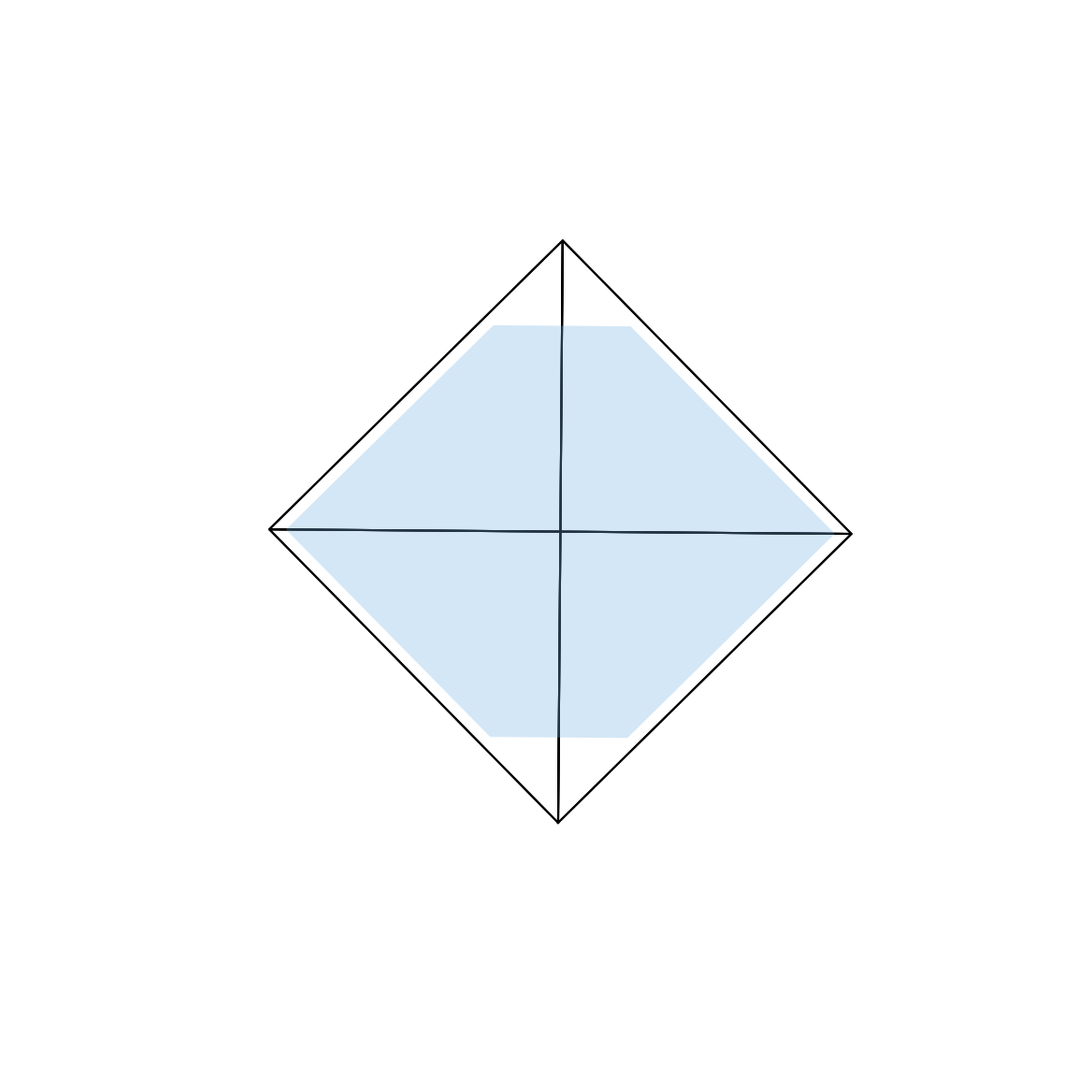}
        \caption{\centering Octahedron\newline $\mu\geq 1.060660171779$}
        \label{fig:octahedron}
    \end{subfigure}
    \vskip\baselineskip
    \begin{subfigure}[b]{0.3\textwidth}
        \centering
        \includegraphics[width=\linewidth, trim={2.1cm 2.1cm 2.1cm 2.1cm}, clip]{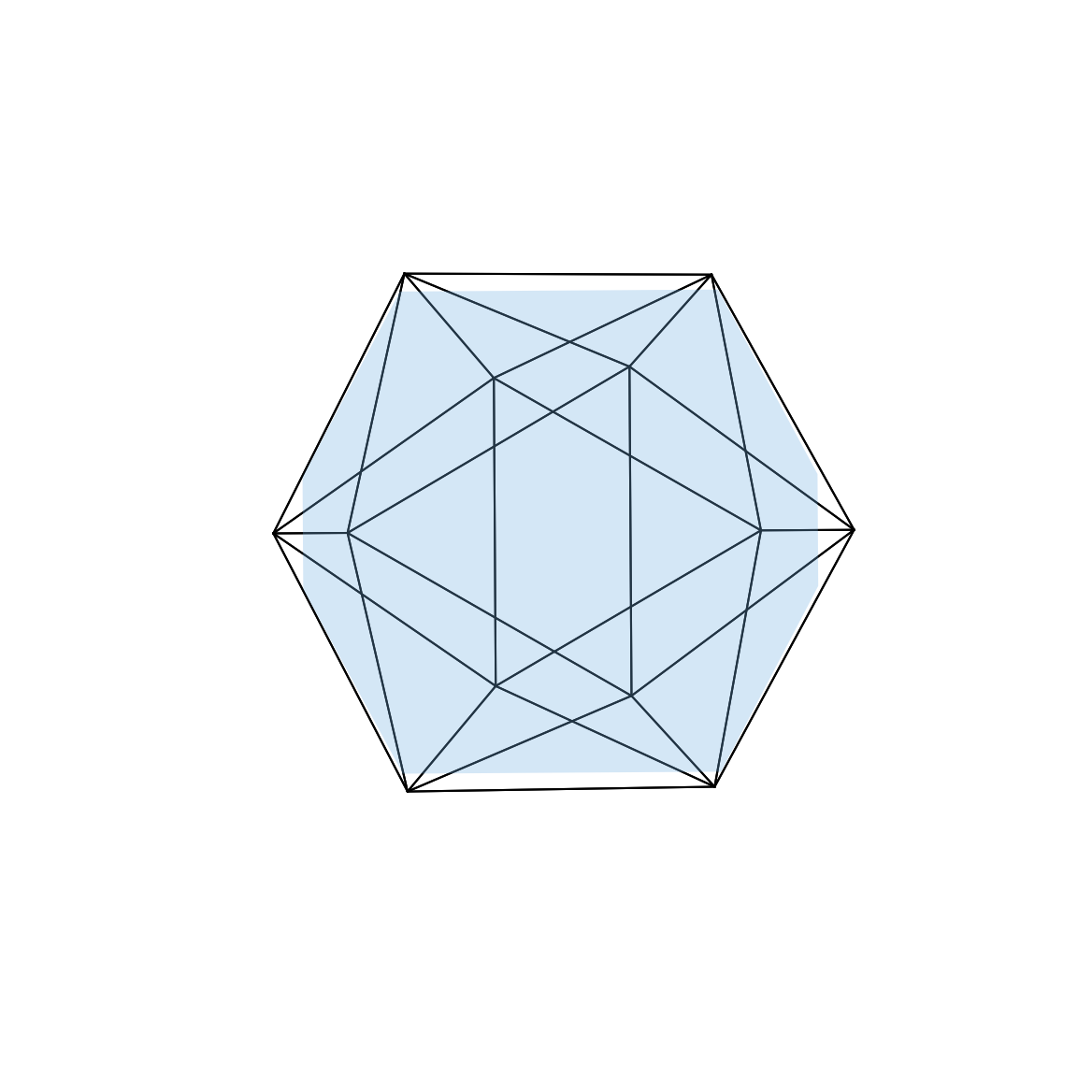}
        \caption{\centering Icosahedron\newline $\mu\geq \mathbf{1.010823060752}$}
        \label{fig:icosahedron}
    \end{subfigure}
    \begin{subfigure}[b]{0.3\textwidth}
        \centering
        \includegraphics[width=\linewidth, trim={2.1cm 2.1cm 2.1cm 2.1cm}, clip]{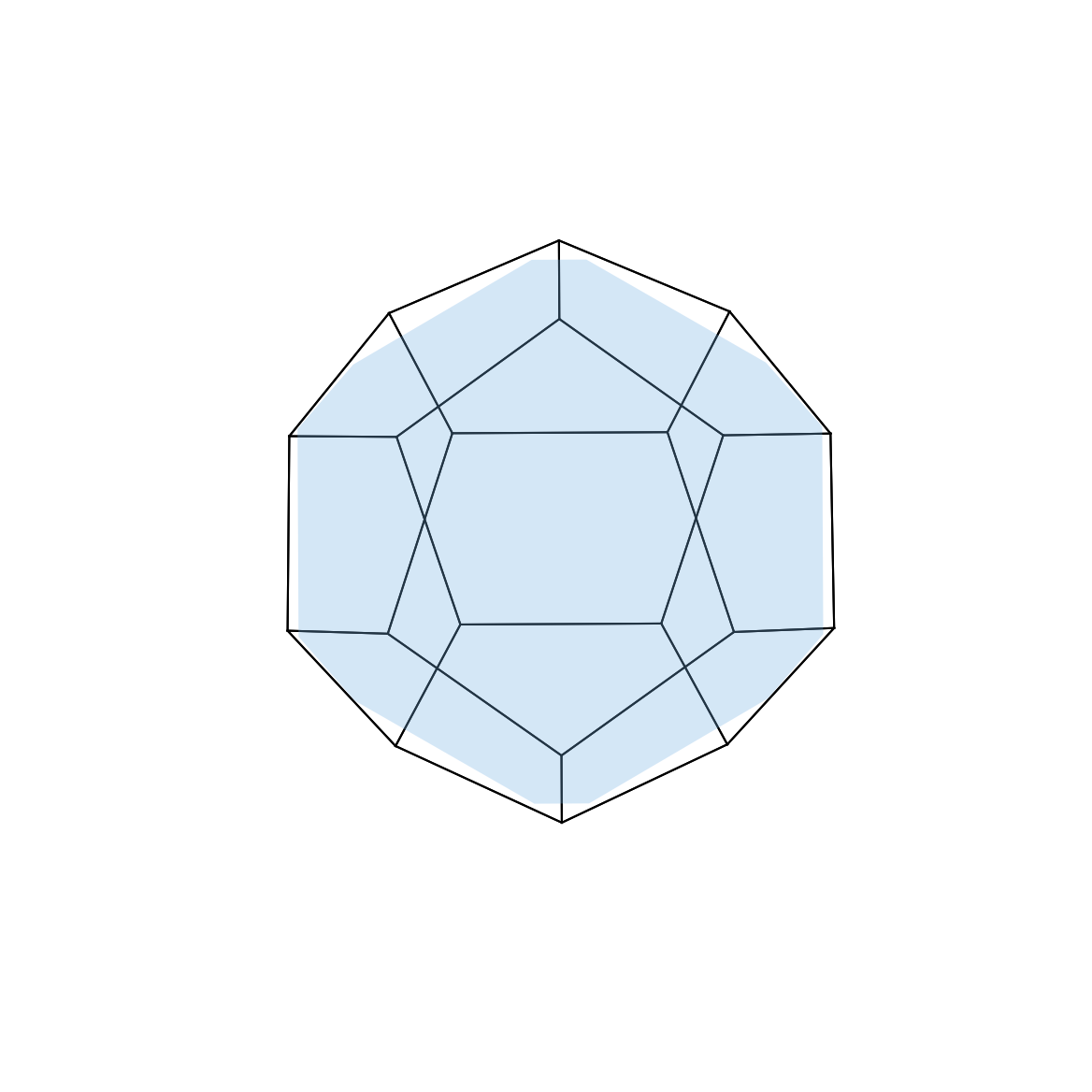}
        \caption{\centering Dodecahedron\newline $\mu\geq \mathbf{1.010823060752}$}
        \label{fig:dodecahedron}
    \end{subfigure}
   \caption{Best-known passages for the five Platonic solids with the blue shaded regions denoting the associated hole in the displayed polyhedron that a rotated copy can pass through (traveling perpendicular to the page). Bolded lower bounds on Nieuwland constants are new to this work. Matching upper bounds for (b)-(e) remain open. Table~\ref{tab:platonic} provides conjectures for their optimal values.}
    \label{fig:platonic}
\end{figure}

\subsection{Our Contributions.}
{\color{blue} Here we take a new approach to computing Nieuwland constants, leveraging ideas from nonsmooth optimization.} Following Fredriksson's suggestion (see~\cite{fredriksson2024}) that the effectiveness of his approach was hindered by numerical precision, our customized method yields local optima with an accuracy of $10^{-12}$. Applying our local search improved the best-known passages for more than half of the Platonic, Archimedean, Catalan, and Johnson solids. No new passages were found for the remaining open cases, indicating a need for tools beyond high-precision local optimization. Our high-precision solutions support new conjectures on simple forms of the optimal passage for the tetrahedron and cuboctahedron. {\color{blue} Section~\ref{sec:methods} formalizes our numerical approach and then Section~\ref{sec:results} presents the resulting experiments and insights.}

\section{Preliminaries and Methodology.} \label{sec:methods}
Throughout, we consider compact, convex polyhedra $P$, without loss of generality having $0\in\mathrm{int\ }P$. For computation, a convex polyhedron can be represented in two equivalent dual forms. One can represent $P$ by its $m$ vertices $v_1,\dots, v_m$ such that $P=\mathtt{convexHull}\{v_i\}_{i=1}^m$ or by its $n$ faces $w_1,\dots, w_n$ such that $P = \{z \mid w_j^Tz \leq 1,\ \text{for}\  j=1,\dots, n\}$. We denote these representations by $\mathtt{Vertices}(P)=\{v_i\}_{i=1}^m$ and $\mathtt{Faces}(P)=\{w_j\}_{j=1}^n$.

For a given polyhedron $P\subseteq\mathbb{R}^3$, Rupert's property is that there exists a polyhedron $Q$ of the same shape and size and a direction $v\in\mathbb{R}^3$ such that
$ P \setminus \{Q + tv \mid t\in\mathbb{R}\}$ is set with a hole.
Steininger and Yurkevich~\cite{steininger2021} provided the following computationally amenable characterization of Rupert's property where seven parameters describe a candidate passage.
\begin{theorem}[Steininger and Yurkevich~\cite{steininger2021}]
A polyhedron $P$ is Rupert if there exist angles $\theta_p,\phi_p$ and $\theta_q,\phi_q$, a rotation $\alpha$, and a translation $(u,v)$
such that
$$ (T_{u,v}\circ R_\alpha \circ M_{\theta_p,\phi_p})(P)\subset \mathrm{int\ }(M_{\theta_q,\phi_q}(P))$$    
where the three maps applied elementwise above are $M_{\theta,\phi}$, denoting a rotation in $\mathbb{R}^3$ followed by projection into $\mathbb{R}^2$ defined as
\begin{equation}\label{m-eq}
M_{\theta, \phi} = \left(
\begin{array}{ccc}
\sin \phi & \sin \theta \cos \phi & \cos \theta \cos \phi\\
0 & \cos \theta & -\sin \theta
\end{array}
\right),
\end{equation}
$R_\alpha$, denoting a rotation in $\mathbb{R}^2$, defined as
\[
R_{\alpha} = \left(
\begin{array}{cc}
\cos \alpha & -\sin \alpha\\
\sin \alpha & \cos \alpha\\
\end{array}
\right),
\]
and $T_{u,v}$, denoting a translation in $\mathbb{R}^2$, defined as
\[
T_{u,v}(x,y) = (x+u, y+v).
\]
\end{theorem}

For notational ease, we collect the above seven parameters defining a passage as a single vector $x=(u,v,\theta_p,\phi_p,\alpha,\theta_q,\phi_q)\in \mathbb{R}^7$. Given a proposed passage $x$, one can define the largest rescaling such that the above necessary condition holds as
$$\mu(x) = \sup\{ \mu \geq 0 \mid \mu \cdot (T_{u,v}\circ R_\alpha \circ M_{\theta_p,\phi_p})(P)\subset M_{\theta_q,\phi_q}(P) \} \ . $$
Then the Nieuwland constant for $P$ is given by
$$ \mu_P = \max_{x\in\mathbb{R}^7} \mu(x)$$
and Rupert's property holds for $P$ if and only if $\max \mu(x)>1$.

Our approach to proving Rupert's property is to locally solve for $\max \mu(x)$ numerically from {\it many} random initializations. This is similar to the approach taken by Fredriksson~\cite{fredriksson2024}, who further reduced the Nieuwland constant to a four-dimensional maximization problem $\max f(\theta_p,\phi_p,\theta_q,\phi_q)$. Evaluating his objective $f$ at given values $(\theta_p,\phi_p,\theta_q,\phi_q)$ corresponds to maximizing a four-dimensional convex quadratic over $mn$ linear inequality constraints. This task is nontrivial. Fredriksson accomplishes this by evaluating the objective at every vertex of the polyhedron in $\mathbb{R}^4$ defined by his $mn$ constraints. However, there can be up to $O(mn^2)$ such vertices over which to compute this maximum (see~\cite{agarwal1998largest}), potentially limiting his approach to only polyhedra with fewer vertices and faces.

In contrast, the objective function that we use, $\mu(x)$, is fairly cheap to compute: The containment $\mu \cdot (T_{u,v}\circ R_\alpha \circ M_{\theta_p,\phi_p})(P)\subset M_{\theta_q,\phi_q}(P)$ requires that each vertex of the two-dimensional polyhedron $\mu \cdot (T_{u,v}\circ R_\alpha \circ M_{\theta_p,\phi_p})(P)$ lies on the correct side of each face of the two-dimensional polyhedron $M_{\theta_q,\phi_q}(P)$. This corresponds to at most $O(mn)$ linear inequalities that $\mu(x)$ must satisfy, giving the following formula for $\mu(x)$ as a finite minimum of at most $O(mn)$ terms:
\begin{align}
    \mu(x) &= \begin{cases}
    \max & \mu \\
    \mathrm{s.t.} & \mu w_j^Tv_i \leq 1 \quad \forall w_j\in\mathtt{Faces}( M_{\theta_q,\phi_q}(P)),\\
    &\qquad\qquad\qquad\ \ v_i \in \mathtt{Vertices}((T_{u,v}\circ R_\alpha \circ M_{\theta_p,\phi_p})(P))
    \end{cases}\nonumber\\
    & = \min\bigg\{\frac{1}{w_j^Tv_i} \mid w_j^Tv_i >0, \ w_j\in\mathtt{Faces}( M_{\theta_q,\phi_q}(P)), \nonumber\\
    &\qquad\qquad\qquad\quad\ \ v_i \in \mathtt{Vertices}((T_{u,v}\circ R_\alpha \circ M_{\theta_p,\phi_p})(P)) \bigg\} \ .
\end{align}
For ease, we denote
$\mu_{i,j}(x) = \begin{cases}
    \frac{1}{w_j^Tv_i} & \text{ if } w_j^Tv_i>0\\
    \infty & \text{otherwise}
\end{cases}
$,
giving $\mu(x) = \min\{\mu_{i,j}(x)\}$. Note that computing the faces of the two-dimensional convex polyhedron $M_{\theta_p,\phi_p}(P)$ can be done efficiently via Graham's scan~\cite{Graham1972}. Hence, despite its three additional parameters, $\mu(x)$'s lower evaluation cost makes it often more amenable to numerically optimize than the objective $f$ considered in~\cite{fredriksson2024}.

\paragraph{A Nonsmooth Trust-Region Method.}
Note $\mu(x)$ is not differentiable everywhere, and in fact, we find it is never differentiable at its local maximizers. Hence, to maximize its value, we use a careful subgradient-type method. Since the nonsmoothness here is entirely due to the finite minimum structure defining $\mu(x) = \min\{\mu_{i,j}(x)\}$, our proposed optimization method linearizes the smooth inner functions $\mu_{i,j}(x)$ at the current iterate $x_k$ and then maximizes the resulting nonsmooth, piecewise linear problem within a trust region $\|x-x_k\|_2 \leq \delta_k$ to produce a search direction. (Note, as $\delta_k$ shrinks, the linearizations $ \mu_{i,j}(x_k) + \nabla \mu_{i,j}(x_k)^T(x-x_k)$ become an increasingly accurate approximation of $\mu_{i,j}(x)$.) 

Formally, at each iteration $k$, we compute a search direction $s_k$ as follows:
\begin{equation}
    \label{eq:subgrad}
    s_k = \mathrm{argmax}_{\|s\|_2\leq \delta_k} \min_{i,j}\{ \mu_{i,j}(x_k) + \nabla \mu_{i,j}(x_k)^Ts\} \ . 
\end{equation}
Then we set the next iterate and trust region size as follows:
$$ \begin{cases}
    x_{k+1} &= x_k + 2^{-n} s_k\\
    \delta_{k+1} &= 2\cdot 2^{-n}\delta_k
\end{cases}$$
where $n\geq 0$ is the smallest integer such that $\mu(x_{k+1}) \geq \mu(x_{k})$. Thus, the trust region size naturally decreases as higher accuracy models are needed to achieve ascent. Note the additional factor of two in the definition of $\delta_{k+1}$ allows trust region sizes to grow over time as well if computed steps $s_k$ are accurate as is (i.e., with $n=0$).
We continue this iteration until the trust region size decreases below a target accuracy, $10^{-12}$ in our experiments. 

This method is very similar to the prox-linear method proposed and studied in~\cite{Drusvyatskiy2019} for more general composite optimization. Here we use an explicit norm bound $\|s\|_2\leq \delta_k$, whereas the authors of~\cite{Drusvyatskiy2019} use a proximal penalty $+\frac{\rho_k}{2}\|s\|_2^2$ in the objective. Considering the first-order optimality conditions of these two approaches verifies they are equivalent with $\rho_k$ acting as a Lagrange multiplier for the constraint $\|s\|_2^2\leq \delta_k^2$.

\paragraph{Computation of Steepest Ascent Directions.}
One can further verify that produced passages $x$ are locally optimal by (numerically) computing the steepest ascent direction of $\mu$ at $x$. To this end, we leverage the following variational description of directional derivatives of $\mu$ at $x\in\mathbb{R}^7$ in direction $v\in\mathbb{R}^7$ of
$$ \mu'(x; v) := \lim_{t\rightarrow 0^+} \frac{\mu(x+tv) - \mu(x)}{t} = \min\{\nabla \mu_{i,j}(x)^Tv \mid \mu_{i,j}(x)=\mu(x)\}\ .$$
Consequently, letting $\Delta=\{\lambda \mid \sum \lambda_{i,j} =1, \lambda\geq 0 \}$ denote the simplex, the steepest ascent direction of $\mu$ at a given $x$ is given by
\begin{align*}
    &\max_{\|v\|_2\leq 1} \min_{i,j}\{\nabla \mu_{i,j}(x)^Tv \mid  \mu_{i,j}(x)=\mu(x)\}\\
    & =\max_{\|v\|_2\leq 1} \min_{\substack{\lambda\in \Delta \\ \lambda_{i,j}(\mu_{i,j}(x)-\mu(x))=0}} \sum_{i,j} \lambda_{i,j} \nabla \mu_{i,j}(x)^Tv\\
    & = \min_{\substack{\lambda\in \Delta \\ \lambda_{i,j}(\mu_{i,j}(x)-\mu(x))=0}} \max_{\|v\|_2\leq 1} \sum_{i,j} \lambda_{i,j} \nabla \mu_{i,j}(x)^Tv\\
    & = \min_{\substack{\lambda\in \Delta \\ \lambda_{i,j}(\mu_{i,j}(x)-\mu(x))=0}} \|\sum_{i,j} \lambda_{i,j} \nabla \mu_{i,j}(x)\|_2,
\end{align*}
where the first equality replaces the finite minimum with the minimum convex combination, the second equality uses the convexity-concavity and compactness to apply a minimax theorem, and the third evaluates the inner maximum. Hence, the steepest ascent direction is given by the smallest convex combination of gradients of currently (numerically) tight $\mu_{i,j}(x)$ functions. This can be computed quickly as a convex quadratic program. Observing $\|\sum_{i,j} \lambda_{ij} \nabla \mu_{ij}(x_k)\|_2$ being nearly zero numerically verifies local optimality. Note this procedure can also be used as a simple alternative method to the above trust-region method, repeatedly updating $x_{k+1} = x_k + 2^{-n} s_k$ with $s_k=\sum_{i,j} \lambda_{ij} \nabla \mu_{ij}(x_k)$.

\section{New Numerically Identified Improved Passages.}\label{sec:results}
We apply the above local optimization procedures to a range of classic, highly symmetric polyhedra studied by prior works, the 5 Platonic, 13 Archimedean, 13 Catalan, and 92 Johnson solids. For each non-Johnson solid, we allocated a computational budget of 48 hours. Each Johnson solid received 16 hours of computation. For shapes where a passage remains unknown in the literature, we instead allocated 180 hours. All computations were done on a desktop with an \texttt{Intel Core i7-12700} processor. Tables~\ref{tab:platonic}-\ref{tab:johnson} summarize the results of repeatedly applying iteration~\eqref{eq:subgrad} to termination from random initializations. 

We sample random initial passages $x_0$ with $\theta_p, \theta_q \sim \text{Uniform}[0, \pi]$, $\phi_p, \phi_q,\alpha \sim \text{Uniform}[0, 2\pi]$, $u, v \sim \text{Uniform}[-0.1, 0.1]$, as well as $u, v = (0, 0)$. 
At each iteration of~\eqref{eq:subgrad}, each needed gradient $\nabla \mu_{i,j}(x_k)$ was computed to machine precision using symbolic differentiation.
Our Python implementation of this procedure and the best numerical passage $x$ found for each shape are available at: \url{https://github.com/RajGosain13/RupertResults}.
As a verification of all stated numerical results, all reported values of our optimized $\mu(x)$ were computed with arbitrary precision arithmetic (using 50 digits of accuracy). Subsequently, the values stated in our tables are all rounded down to 12 digits.

\subsection{Improvements for Platonic Solids.}
Beyond the cube, exact Nieuwland constants for the remaining Platonic solids are not known. Repeatedly applying our iteration~\eqref{eq:subgrad} derived improvements of varying magnitudes to prior suboptimal bounds on the constant for all other Platonic solids. These new best-known passages, shown in Figure~\ref{fig:platonic}, provide clear analytic candidates for their optimal values, shown in Table~\ref{tab:platonic}. The best passage we identified for the tetrahedron improved the prior best bound by approximately $2\times 10^{-4}$. Upon seeing this new constant agrees closely with a simple irrational ratio, we derived the following matching explicit passage.
\begin{theorem}
    The Nieuwland constant for the tetrahedron is at least  $\frac{\sqrt{6}}{1+\sqrt{2}}$.
\end{theorem}
\begin{proof}
Consider the following pair of (differently rotated and translated) tetrahedra, each with all edge lengths $\sqrt{3}$, defined by a convex hull of their extreme points:
\begin{align*}
    T_1 = \mathtt{convexHull}\{ &\left(0,\ 1,\ -\sqrt{1/8}\right),\quad  \left(\sqrt{3/4},\ -1/2,\ -\sqrt{1/8}\right),\\
    &\left(-\sqrt{3/4},\ -1/2,\ -\sqrt{1/8}\right),\quad  \left(0,\ 0,\ \sqrt{9/8}\right)\} \ , \\
    T_2 = \mathtt{convexHull}\{ & \left((2-\sqrt{2})/4,\  (5\sqrt{6} - 2\sqrt{3})/12,\  (2+\sqrt{2})/4\right),\\
    & \left(-(2-\sqrt{2})/4,\  (5\sqrt{6} - 2\sqrt{3})/12,\  -(2+\sqrt{2})/4\right),\\
    &\left((2+\sqrt{2})/4,\ -(2\sqrt{3}+\sqrt{6})/12,\ -(2-\sqrt{2})/4\right),\\
    & \left(-(2+\sqrt{2})/4,\ -(2\sqrt{3}+\sqrt{6})/12,\ (2-\sqrt{2})/4\right)\} \ . 
\end{align*}
Simple algebra can verify that when projected into the $x,y$-plane, $\mu=\frac{\sqrt{6}}{1+\sqrt{2}}$ times each extreme point of $T_2$ lies in the projection of $T_1$ into the $x,y$-plane. 
\end{proof}
We conjecture this is the tetrahedron's exact Nieuwland constant.
\begin{conjecture}\label{conj:tetrahedron}
    The Nieuwland constant for the tetrahedron is equal to  $\frac{\sqrt{6}}{1+\sqrt{2}}$.
\end{conjecture}
Note the conjectured optimal value $\rho$ for the icosahedron and dodecahedron in Table~\ref{tab:platonic} is due to Steininger and Yurkevich \cite{Steininger2023}.
They further conjectured that if a Platonic solid is Rupert, then so is its dual with exactly the same constant.
\begin{conjecture}[Steininger and Yurkevich~\cite{Steininger2023}] \label{conj:dualConstants}
The Nieuwland constant for each Platonic solid is equal to the constant for its dual Platonic solid. 
\end{conjecture}
Our higher precision bounds provide further support for this conjecture. From Table~\ref{tab:platonic}, the constants for the cube and octahedron, icosahedron and dodecahedron, and the tetrahedron and itself all agree with their dual's value to at least 12 digits.

\begin{table}[t]
\centering
\begin{tabular}{|c|c|c|c|c|}
\hline
Platonic Solid & Hours & Prior Best $\mu$ & Best $\mu$ Seen & Conjectured Optimal \\
\hline
Tetrahedron & 48 & 1.014473\cite{steininger2021} & {\bf 1.01461187233467} & $\frac{\sqrt{6}}{1+\sqrt{2}}$ \\
Cube & 0 & $\frac{3\sqrt{2}}{4}$\cite{schreck1950} & - & $\frac{3\sqrt{2}}{4}$ \\
Octahedron & 48 & 1.060660\cite{steininger2021} & {\it 1.06066017177981} & $\frac{3\sqrt{2}}{4}$ \\
Icosahedron & 48 & 1.010805\cite{steininger2021} & {\bf 1.01082306075264} & $\rho$ \\
Dodecahedron & 48 & 1.010818\cite{steininger2021} & {\bf 1.01082306075208} & $\rho$ \\
\hline
\end{tabular}
\caption{Results of repeated trials optimizing passages for each Platonic solid via~\eqref{eq:subgrad} for forty-eight hours each. Improvements on known lower bounds for their Nieuwland constant are denoted by bold and accuracy improvements in italics. Here $\rho$ denotes the smallest positive root of $2025x^8 - 11970x^6 + 17009x^4 - 9000x^2 + 2000$.}
\label{tab:platonic}
\end{table}

        \begin{figure}
    \centering
    \begin{subfigure}[b]{0.19\textwidth}
        \centering
        \includegraphics[width=\linewidth, trim={2.1cm 2.1cm 2.1cm 2.1cm}, clip]{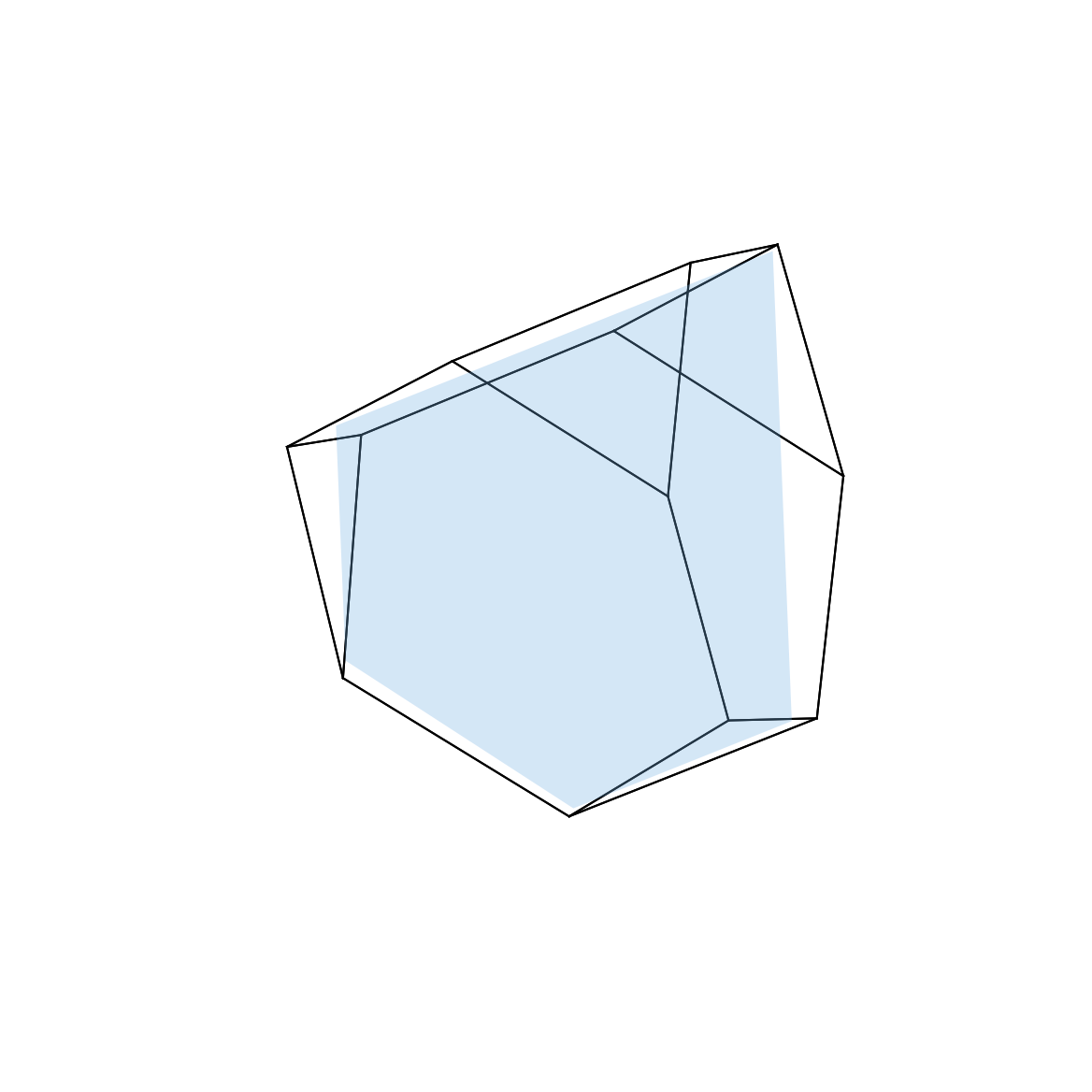}
        \caption{\centering Truncated Tetrahedron}
    \end{subfigure}
    \hfill
    \begin{subfigure}[b]{0.19\textwidth}
        \centering
        \includegraphics[width=\linewidth, trim={2.1cm 2.1cm 2.1cm 2.1cm}, clip]{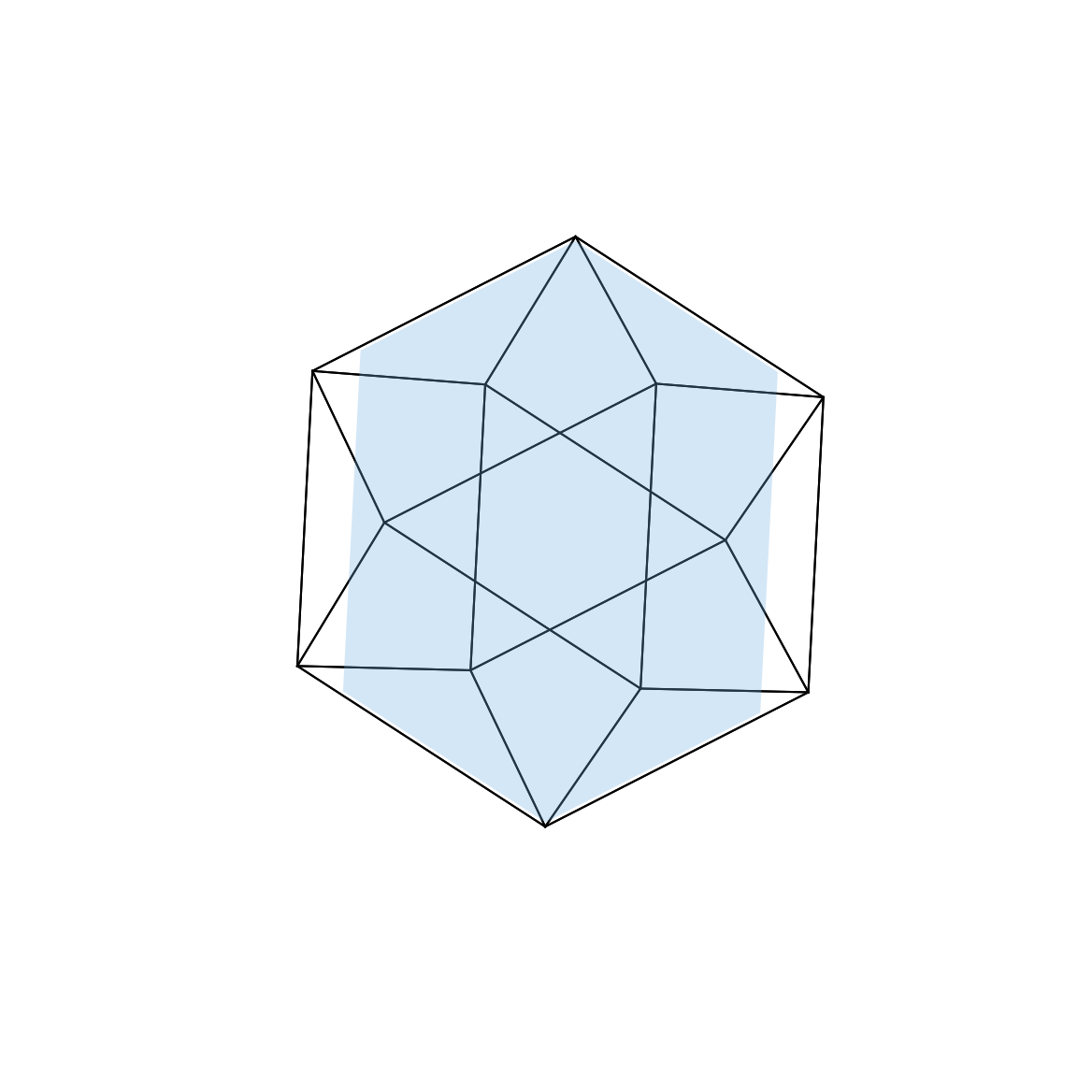}
        \caption{\centering Cuboctahedron}
    \end{subfigure}
    \hfill
    \begin{subfigure}[b]{0.19\textwidth}
        \centering
        \includegraphics[width=\linewidth, trim={2.1cm 2.1cm 2.1cm 2.1cm}, clip]{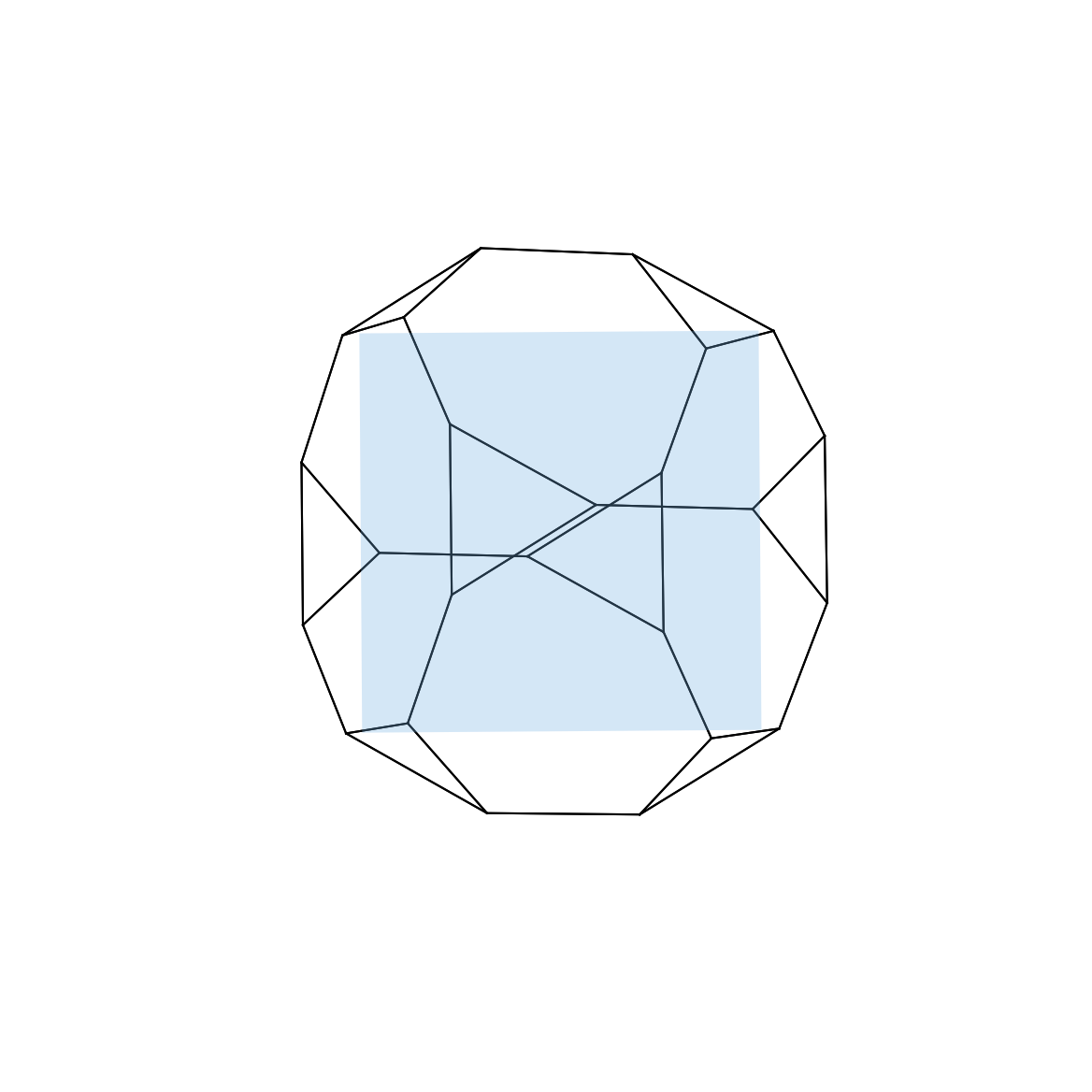}
        \caption{\centering Truncated Cube}
    \end{subfigure}
    \vskip\baselineskip\vskip-9pt
    \hfill
    \begin{subfigure}[b]{0.19\textwidth}
        \centering
        \includegraphics[width=\linewidth, trim={2.1cm 2.1cm 2.1cm 2.1cm}, clip]{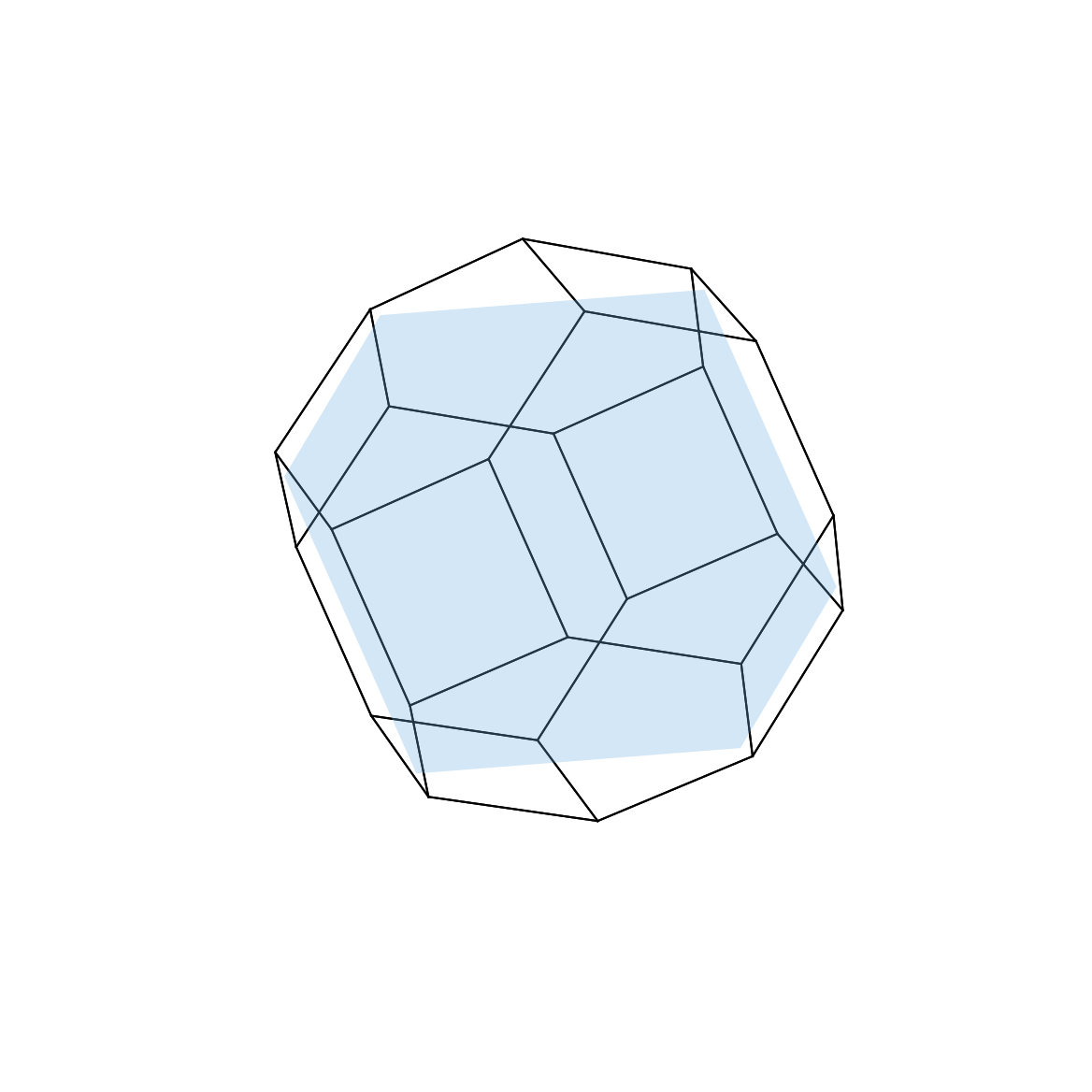}
        \caption{\centering Truncated Octahedron}
    \end{subfigure}
    \hfill
    \begin{subfigure}[b]{0.19\textwidth}
        \centering
        \includegraphics[width=\linewidth, trim={2.1cm 2.1cm 2.1cm 2.1cm}, clip]{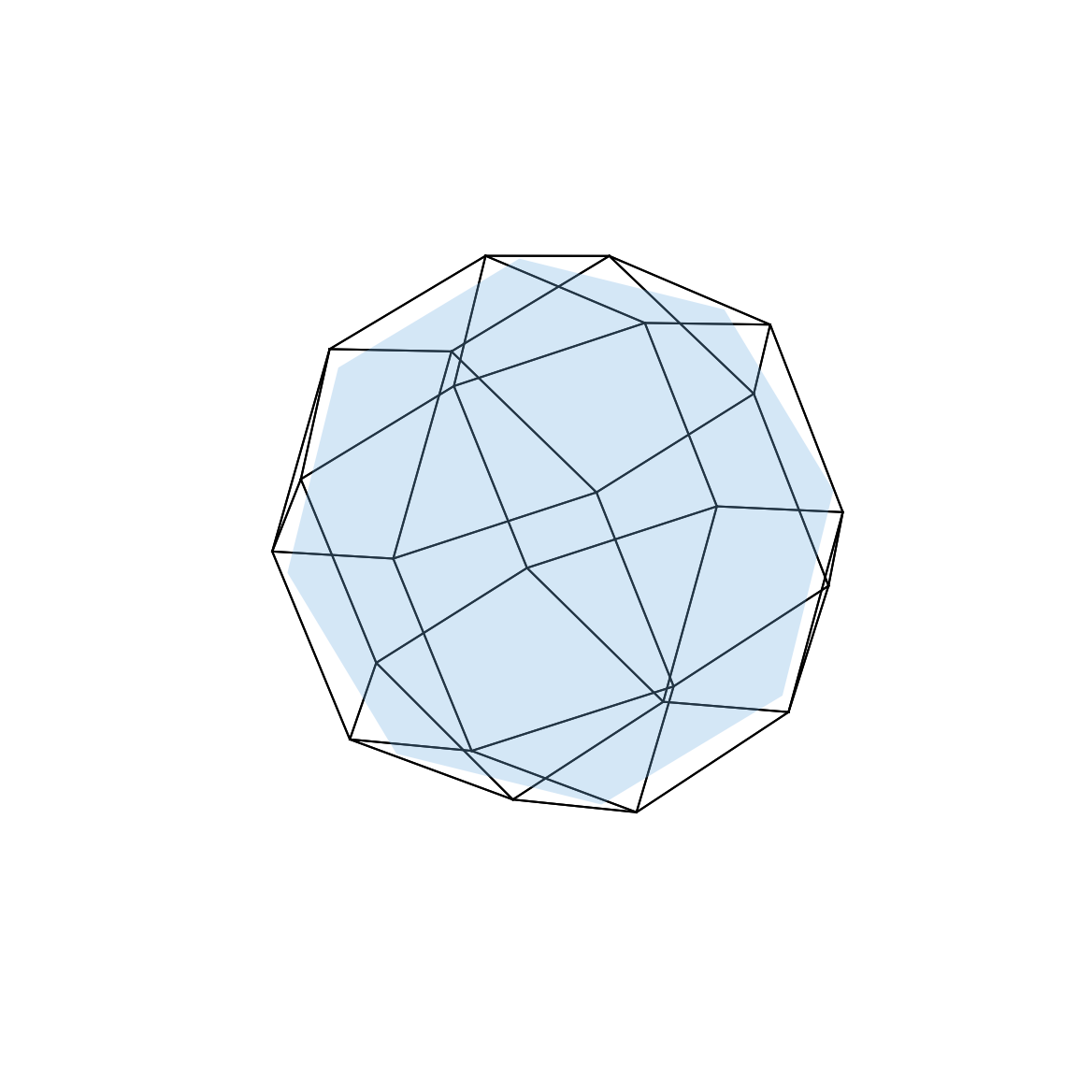}
        \caption{\centering Rhombicuboctahedron}
    \end{subfigure}
    \hfill \ 
    \vskip\baselineskip\vskip-9pt
    \begin{subfigure}[b]{0.19\textwidth}
        \centering
        \includegraphics[width=\linewidth, trim={2.1cm 2.1cm 2.1cm 2.1cm}, clip]{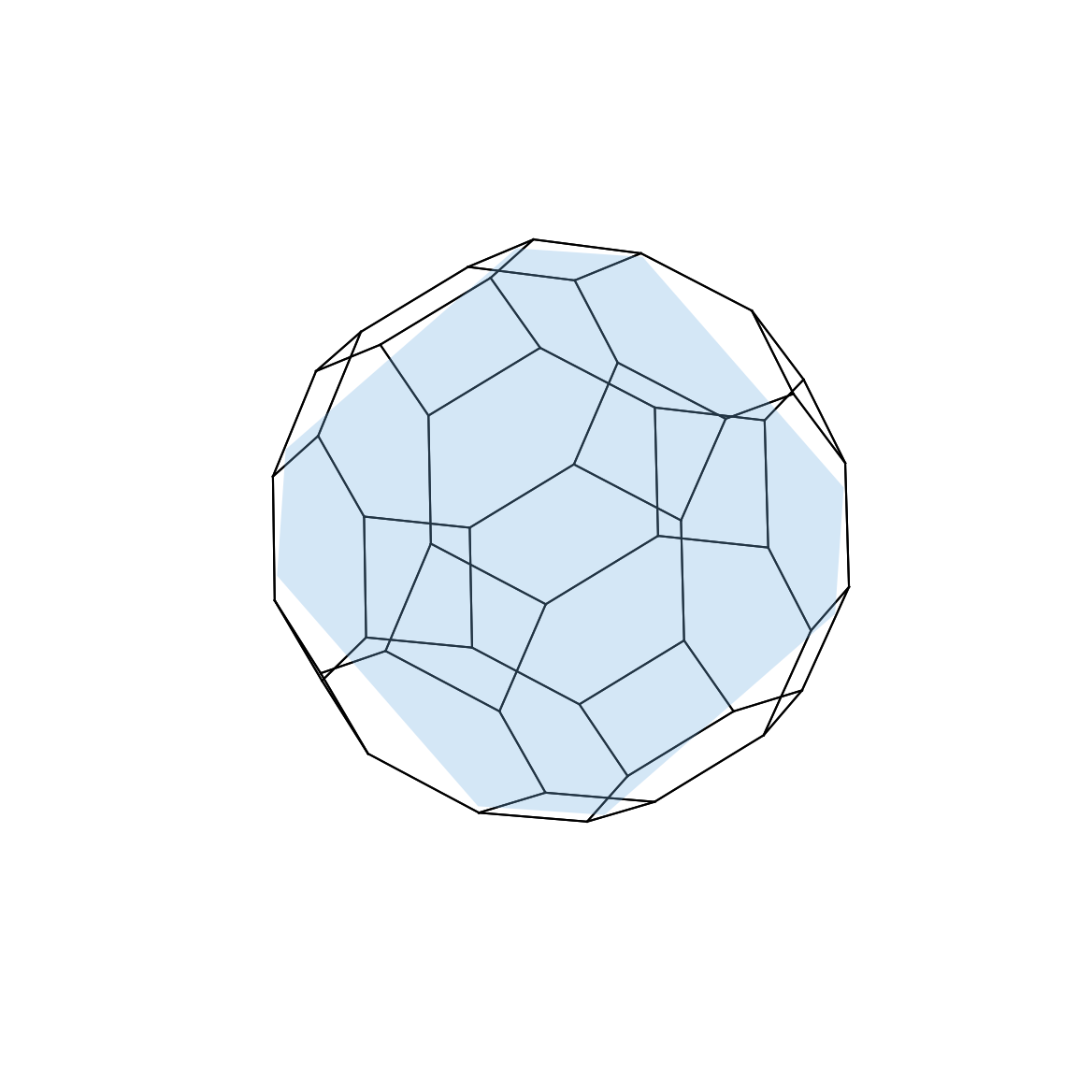}
        \caption{\centering Truncated Cuboctahedron}
    \end{subfigure}
    \hfill
    \begin{subfigure}[b]{0.19\textwidth}
        \centering
        \includegraphics[width=\linewidth, trim={2.1cm 2.1cm 2.1cm 2.1cm}, clip]{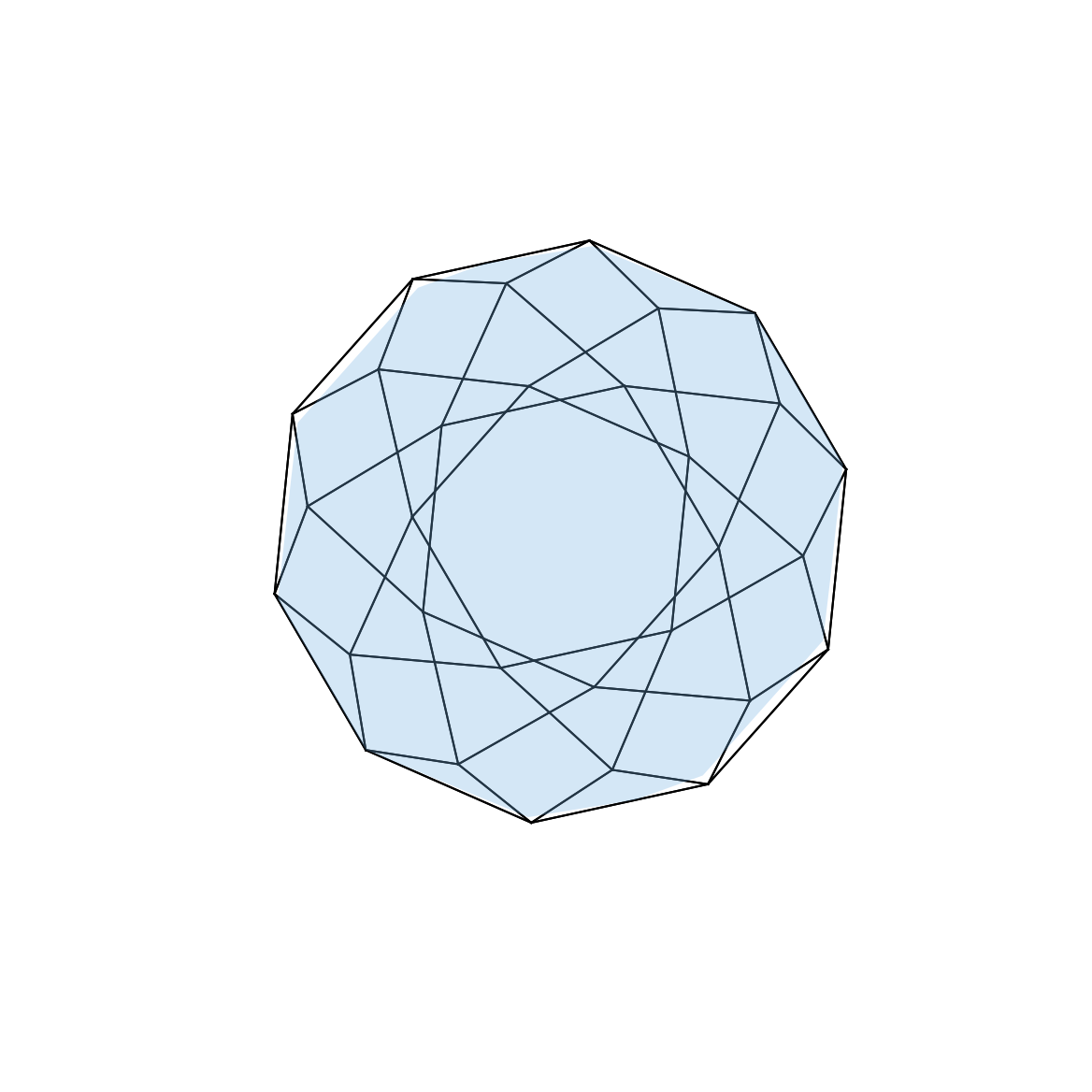}
        \caption{\centering Icosidodecahedron}
    \end{subfigure}
    \hfill
    \begin{subfigure}[b]{0.19\textwidth}
        \centering
        \includegraphics[width=\linewidth, trim={2.1cm 2.1cm 2.1cm 2.1cm}, clip]{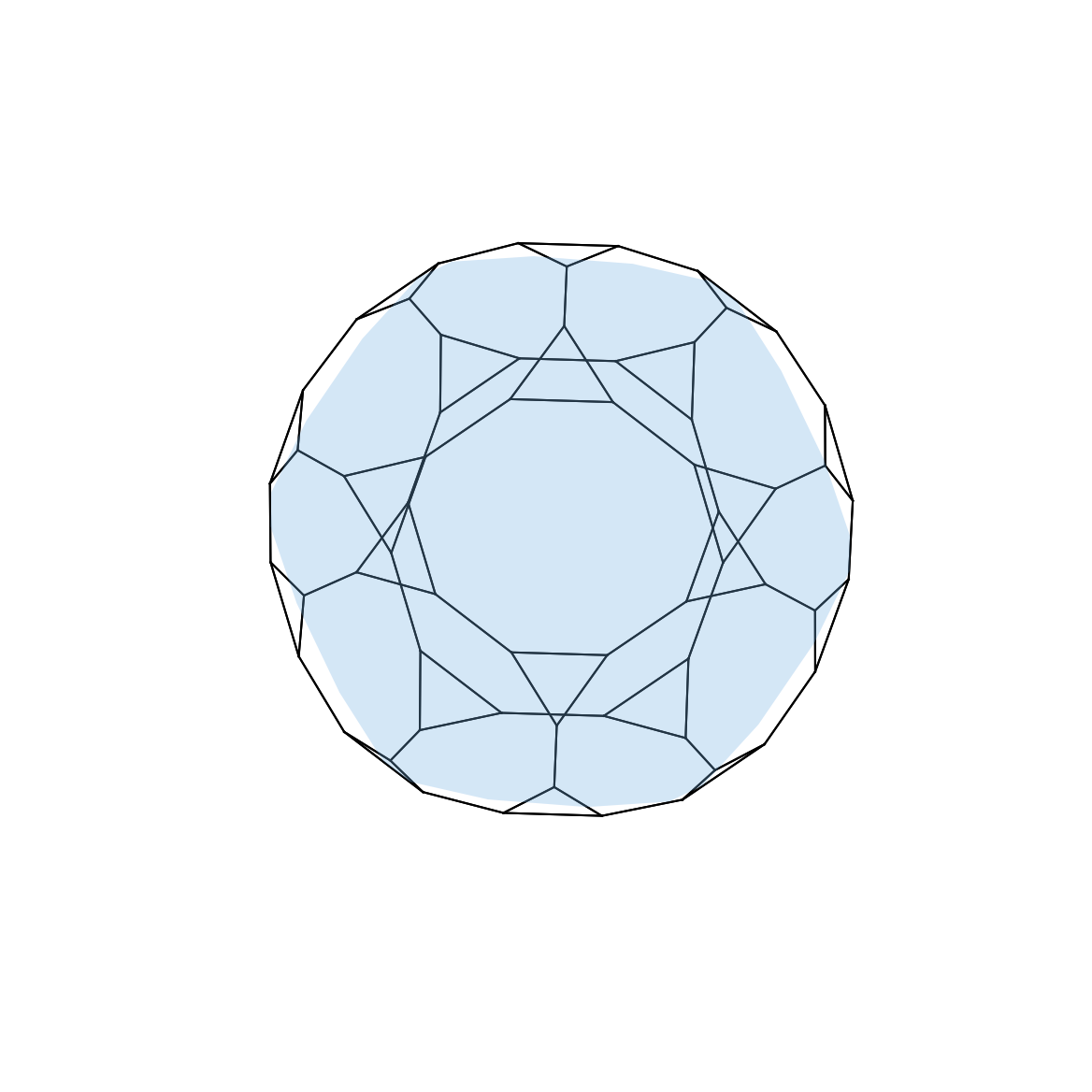}
        \caption{\centering Truncated Dodecahedron}
    \end{subfigure}
    \vskip\baselineskip\vskip-9pt
    \hfill 
    \begin{subfigure}[b]{0.19\textwidth}
        \centering
        \includegraphics[width=\linewidth, trim={2.1cm 2.1cm 2.1cm 2.1cm}, clip]{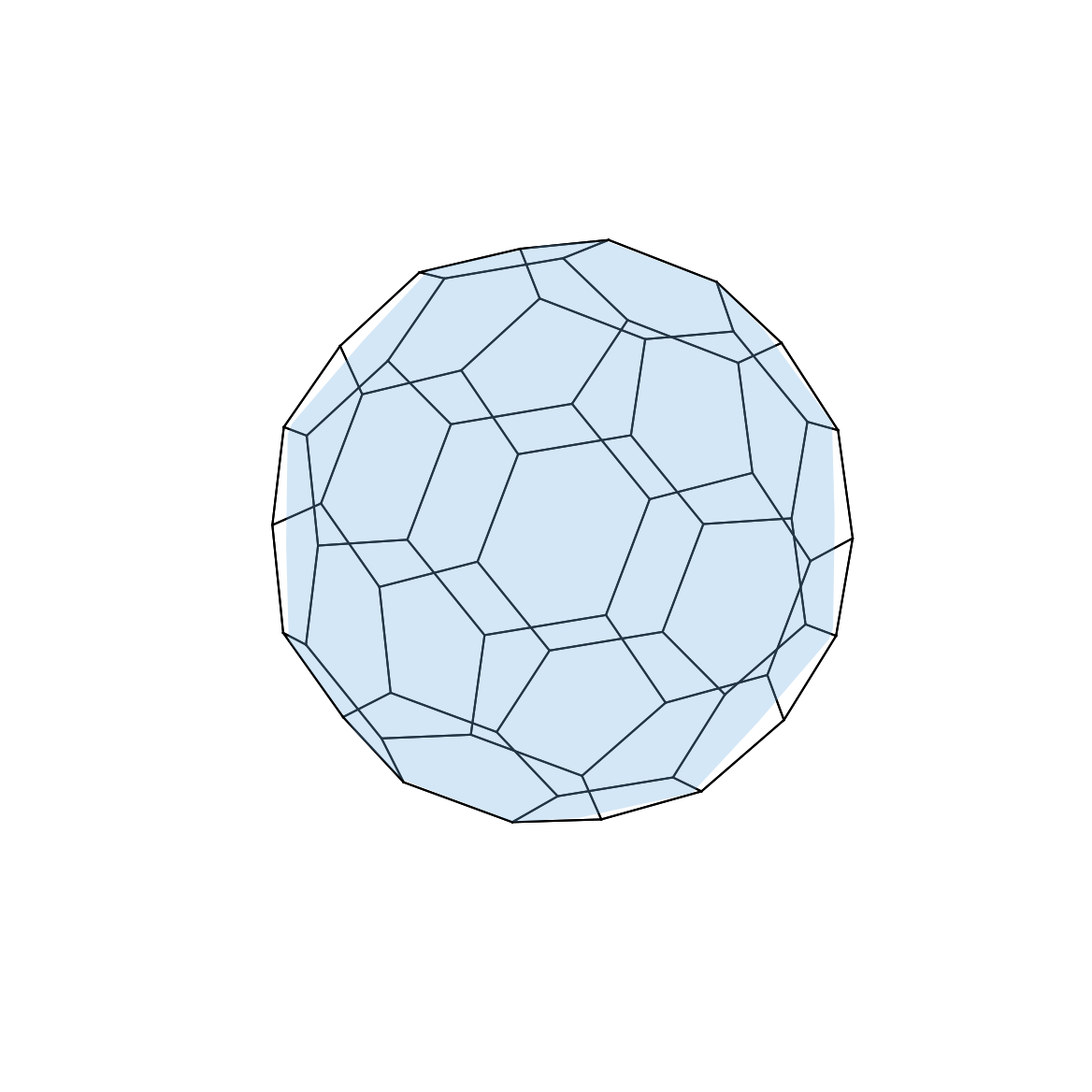}
        \caption{\centering Truncated Icosahedron}
    \end{subfigure}
    \hfill
    \begin{subfigure}[b]{0.19\textwidth}
        \centering
        \includegraphics[width=\linewidth, trim={2.1cm 2.1cm 2.1cm 2.1cm}, clip]{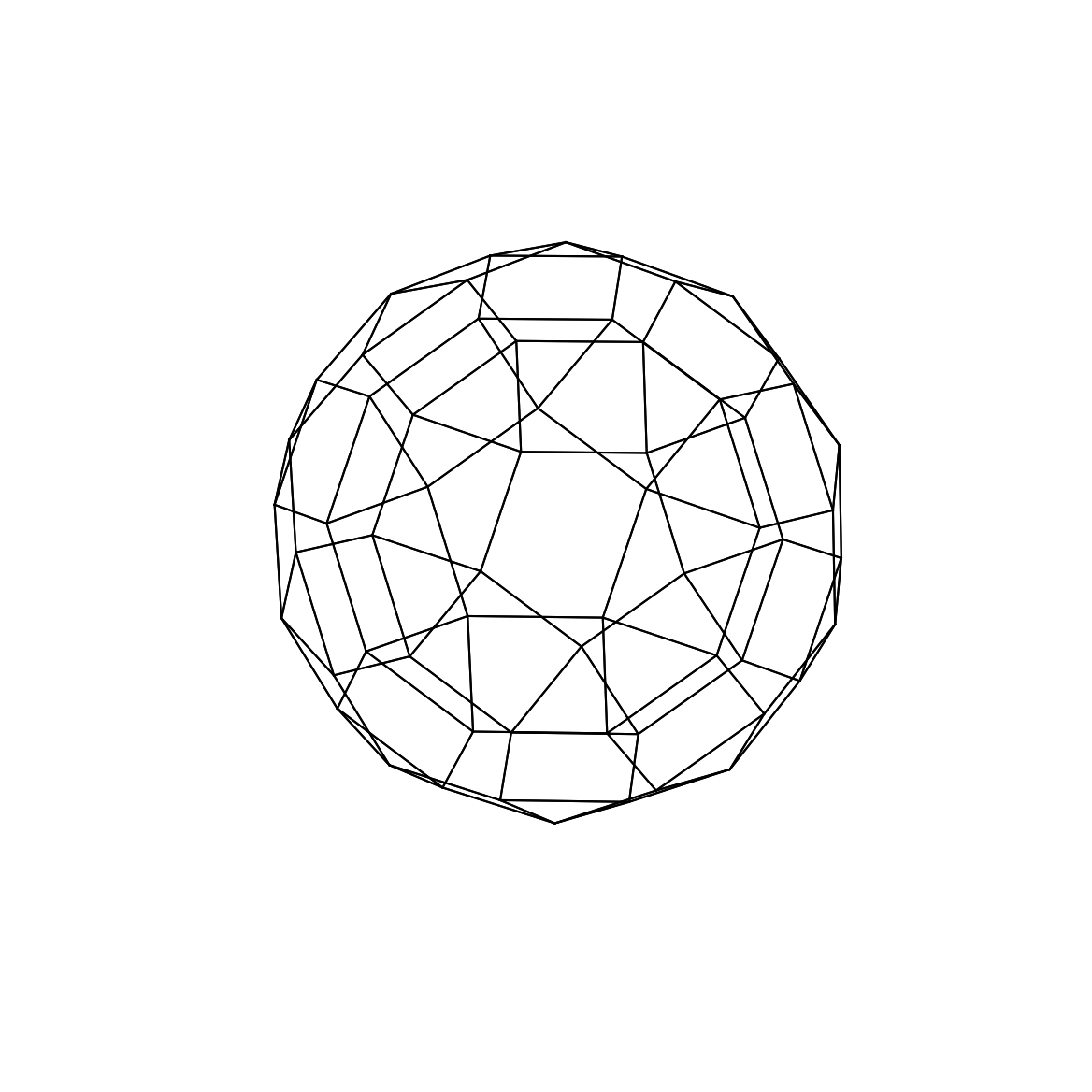}
        \caption{\centering Rhombicosidodecahedron}
    \end{subfigure}
    \hfill \ 
    \vskip\baselineskip\vskip-9pt
    \begin{subfigure}[b]{0.19\textwidth}
        \centering
        \includegraphics[width=\linewidth, trim={2.1cm 2.1cm 2.1cm 2.1cm}, clip]{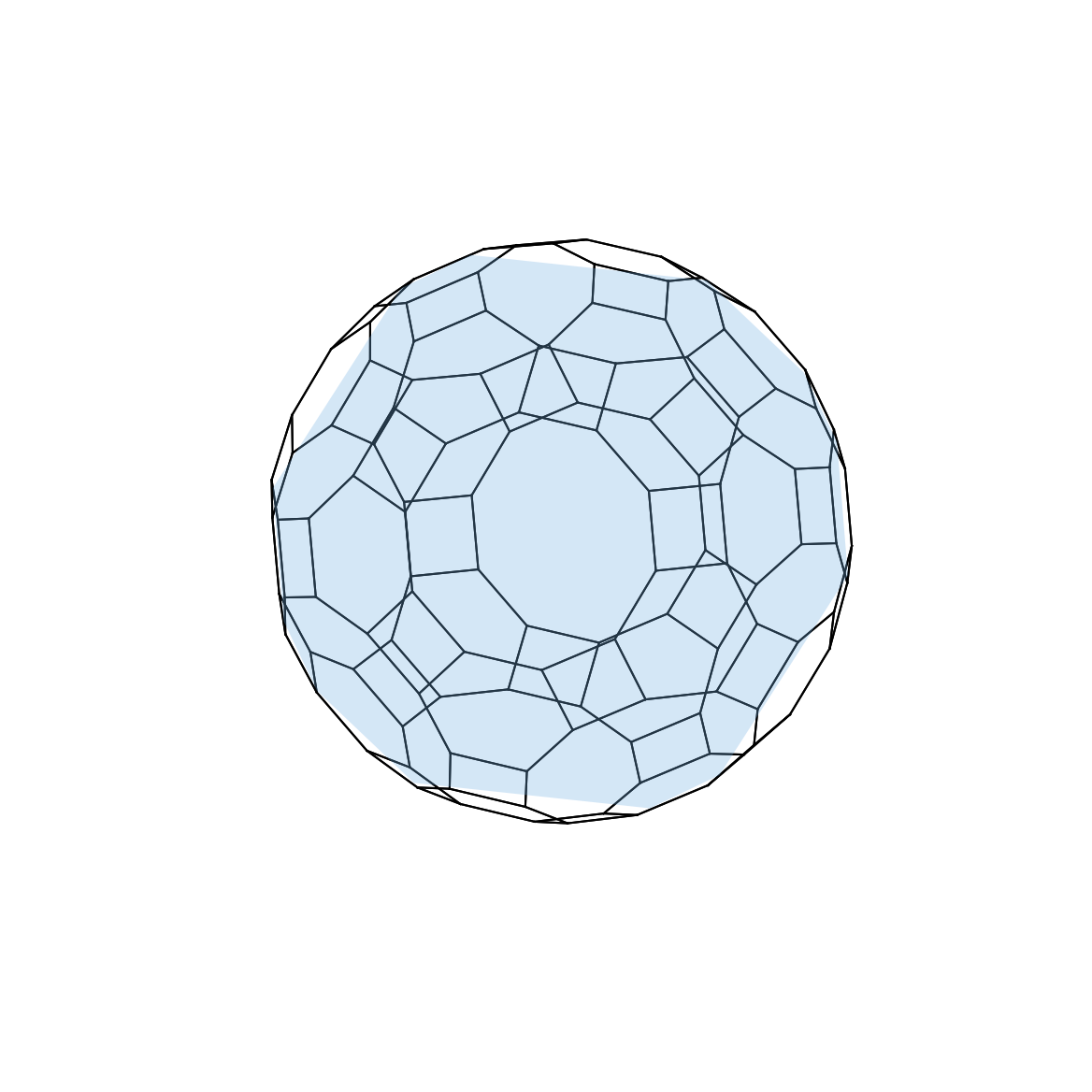}
        \caption{\centering Truncated Icosidodecahedron}
    \end{subfigure}
    \hfill
    \begin{subfigure}[b]{0.19\textwidth}
        \centering
        \includegraphics[width=\linewidth, trim={2.1cm 2.1cm 2.1cm 2.1cm}, clip]{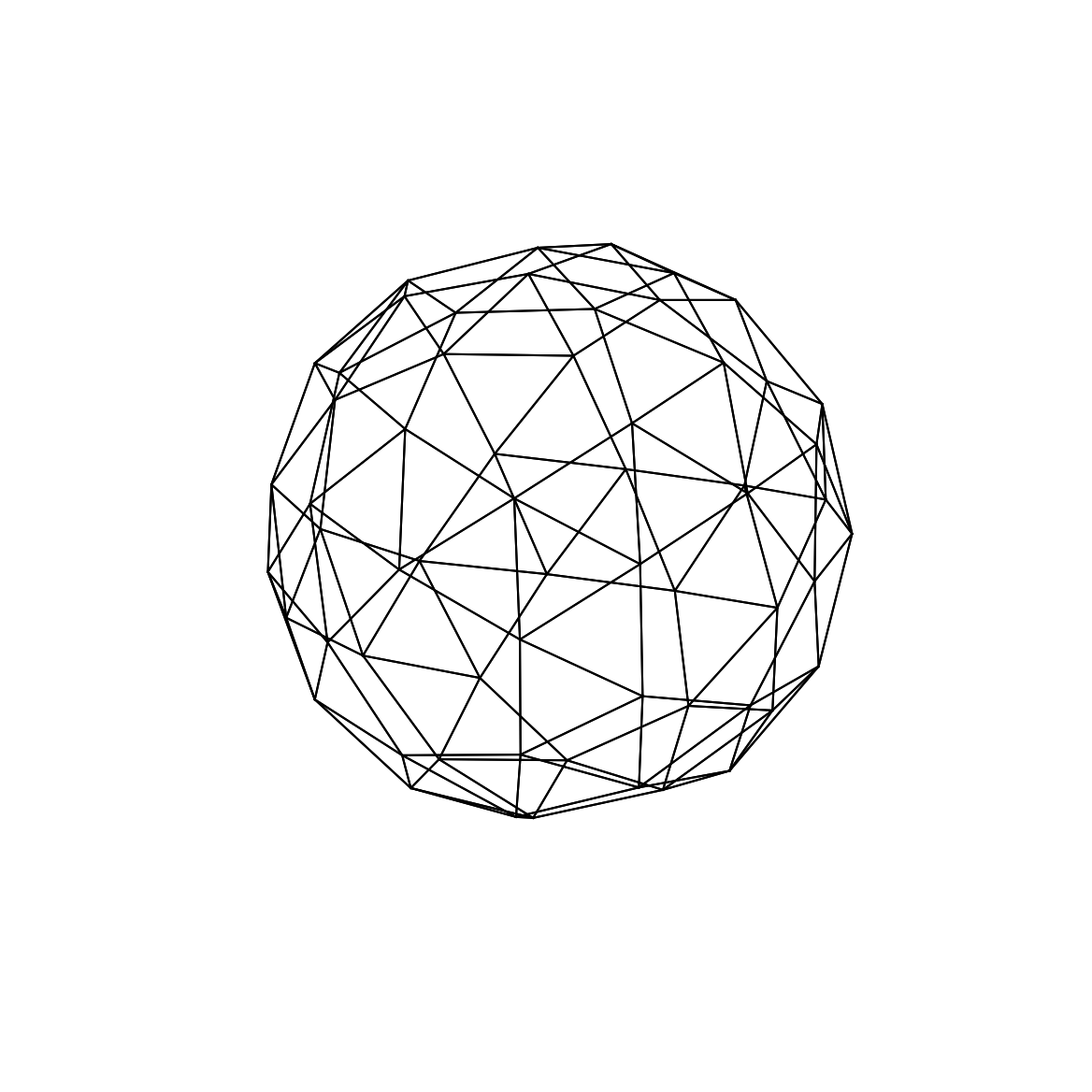}
        \caption{\centering Snub Dodecahedron}
    \end{subfigure}
    \hfill
    \begin{subfigure}[b]{0.19\textwidth}
        \centering
        \includegraphics[width=\linewidth, trim={2.1cm 2.1cm 2.1cm 2.1cm}, clip]{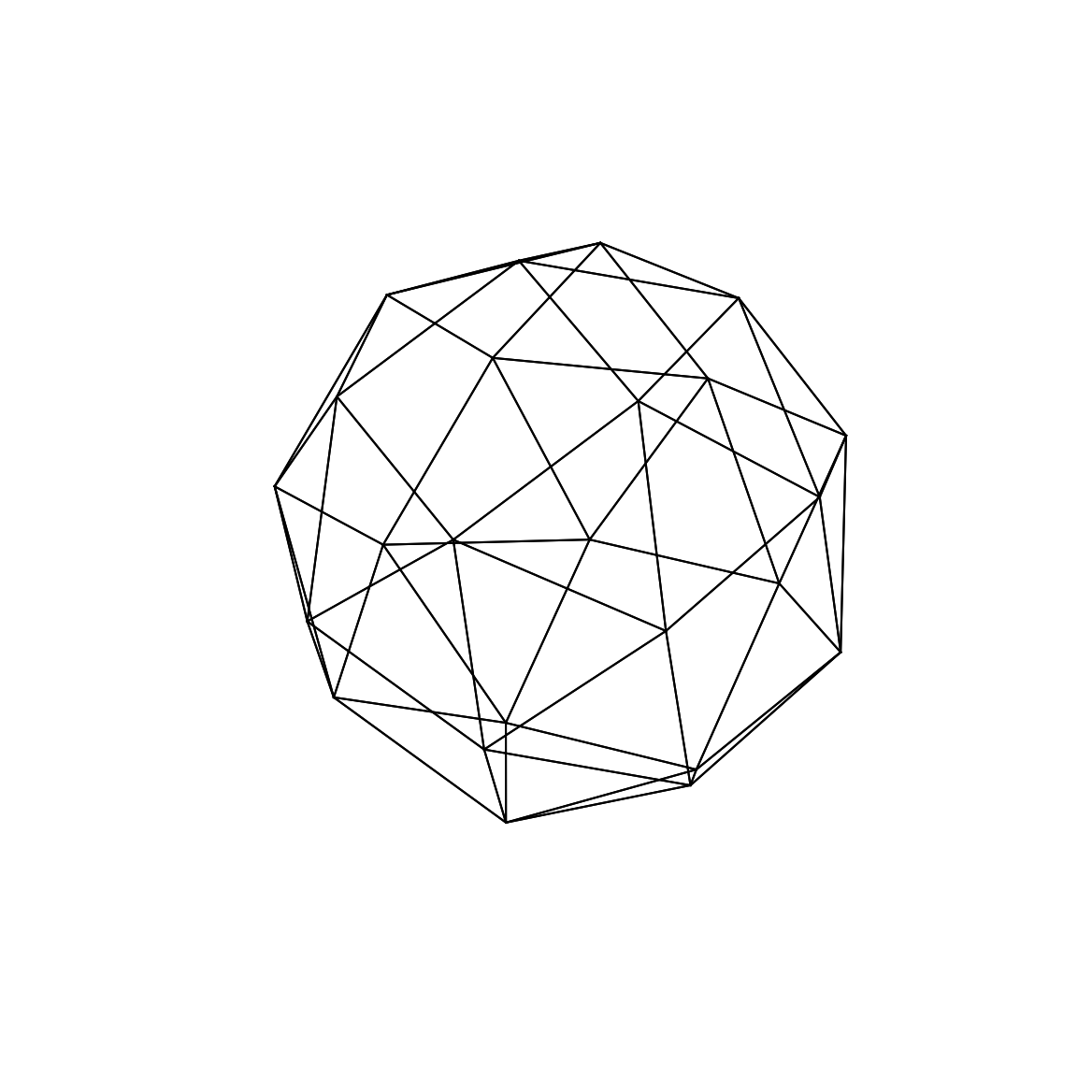}
        \caption{\centering Snub Cube}
    \end{subfigure}
   \caption{(New) Best-known, highly optimized passages for the thirteen Archimedean solids and bounds on their Nieuwland constants. See Table~\ref{tab:archimedean} for exact numerically computed constants.}
    \label{fig:archimedean}
\end{figure}

\subsection{Improvements for Archimedean Solids.}
Applying our method to the Archimedean solids improved the best-known lower bounds for five solids previously known to be Rupert. Our higher precision bounds could be used to motivate conjectures on the Nieuwland constants for the Archimedean solids. For example, the best cuboctahedron passage found matches exactly the best tetrahedron passage we found, indicating its Nieuwland constant may also be $\frac{\sqrt{6}}{1+\sqrt{2}}$. See Table~\ref{tab:archimedean}. Applied to rhombicosidodecahedron, snub dodecahedron, and snub cube for 180 hours each, our method was unable to find passages that would prove Rupertness. This provides further support to the (negative) conjecture of Steininger and Yurkevich.
\begin{conjecture}[Steininger and  Yurkevich~\cite{steininger2021}]
The rhombicosidodecahedron, snub dodecahedron,  and snub cube are not Rupert. 
\end{conjecture}
Given the computational budget and accuracies we leveraged, we expect techniques beyond high-precision local optimization will be needed if such passages exist.

\begin{table}[t]
\centering
\begin{tabular}{|c|c|c|c|}
\hline
Archimedean Solid & Hours & Prior Best $\mu$ & Our Best $\mu$\\
\hline
Truncated Tetrahedron & 48 & 1.014210\cite{steininger2021} & \bf{1.014255711995} \\
Cuboctahedron & 48 & 1.01461\cite{steininger2021} & \it{1.014611872354} \\
Truncated Cube & 48 & 1.030659\cite{steininger2021} & \bf{1.030661650181} \\
Truncated Octahedron & 48 & 1.014602\cite{steininger2021} & 1.014566571546 \\
Rhombicuboctahedron & 48 & 1.012819\cite{steininger2021} & 1.012785597490 \\
Truncated Cuboctahedron & 48 & 1.006563\cite{steininger2021} & 1.006561784960 \\
Icosidodecahedron & 48 & 1.000878\cite{steininger2021} & \bf{1.000885427887} \\
Truncated Dodecahedron & 48 & 1.001612\cite{steininger2021} & \bf{1.001614361787} \\
Truncated Icosahedron & 48 & 1.001955\cite{steininger2021} & \bf{1.001965186189}\\
Rhombicosidodecahedron & 180 & --- & 0.999999999999 \\
Truncated Icosidodecahedron & 48 & 1.002048\cite{steininger2021} & 1.002046507167 \\
Snub Dodecahedron & 180 & --- & 0.999999999999 \\
Snub Cube & 180 & --- & 0.999999999999 \\
\hline
\end{tabular}
\caption{Results of repeated trials optimizing passages for each Archimedean solid via~\eqref{eq:subgrad} for 48 hours each. Improvements on known lower bounds for their Nieuwland constant are denoted by bold, accuracy improvements in italics.}
\label{tab:archimedean}
\end{table}

\subsection{Improvements for Catalan Solids and Johnson Solids.}
\begin{figure}
    \centering
    
        \begin{subfigure}[b]{0.19\textwidth}
        \centering
        \includegraphics[width=\linewidth, trim={2.1cm 2.1cm 2.1cm 2.1cm}, clip]{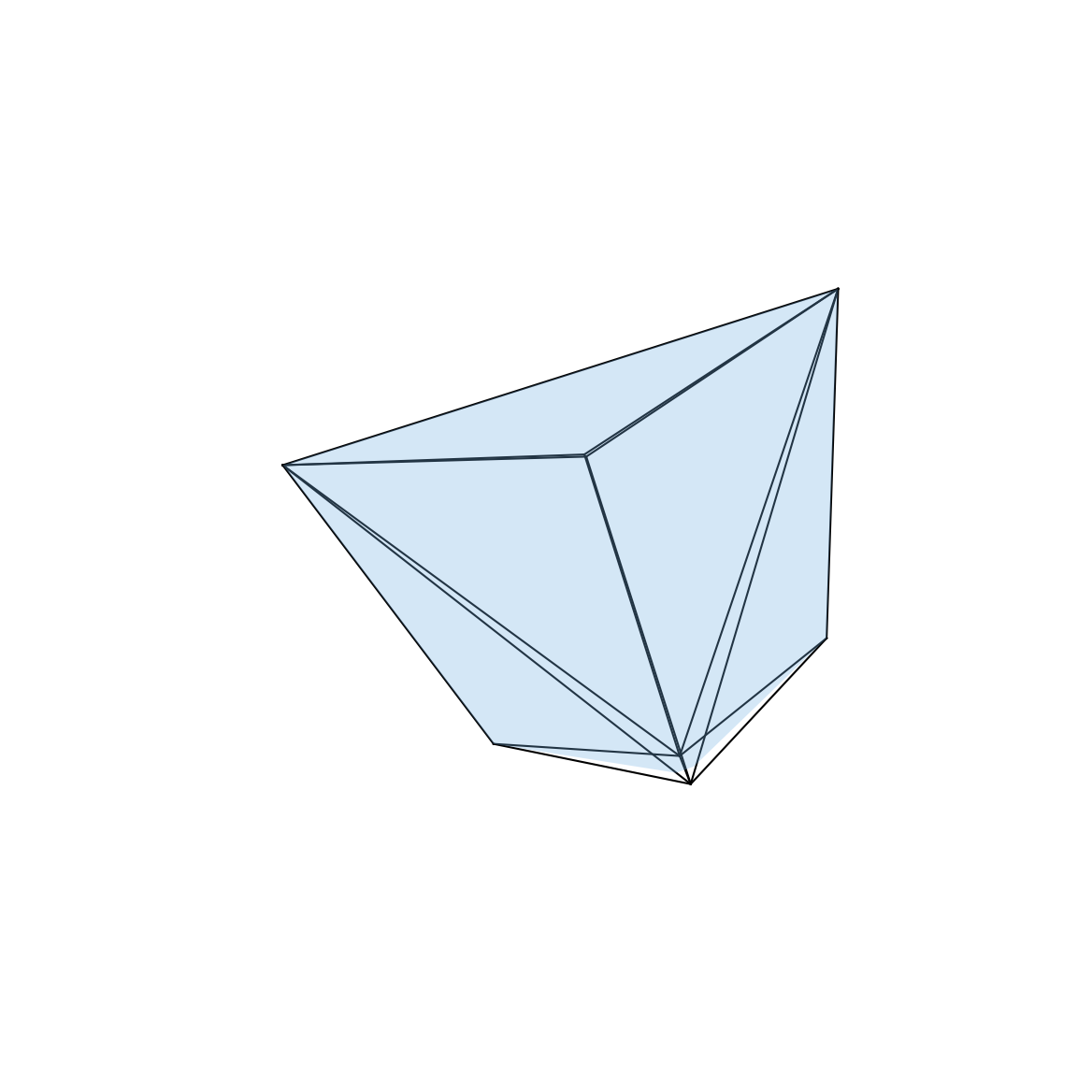}
        \caption{\centering Triakis Tetrahedron}
    \end{subfigure}
    \hfill
    \begin{subfigure}[b]{0.19\textwidth}
        \centering
        \includegraphics[width=\linewidth, trim={2.1cm 2.1cm 2.1cm 2.1cm}, clip]{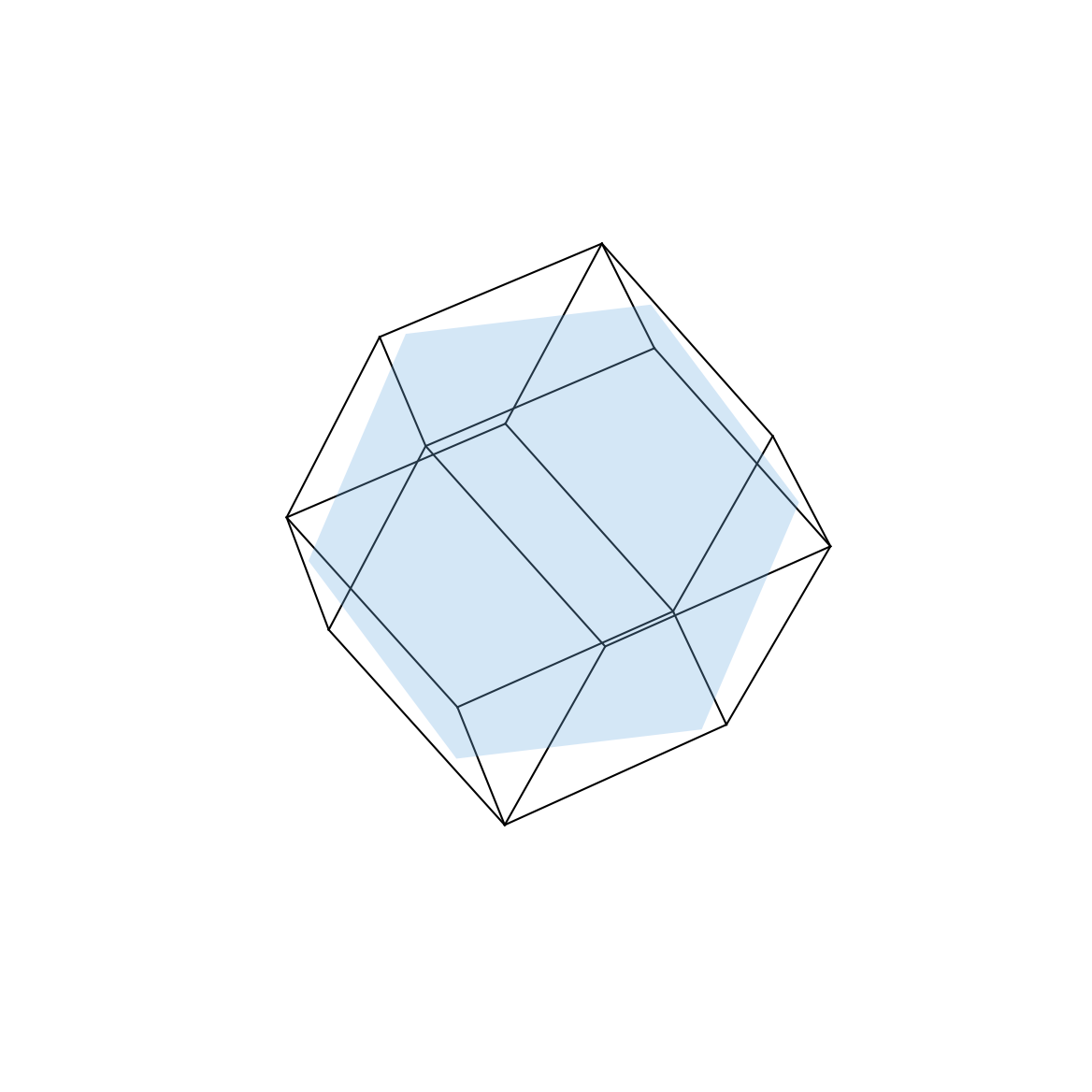}
        \caption{\centering Rhombic Dodecahedron}
    \end{subfigure}
    \hfill
    \begin{subfigure}[b]{0.19\textwidth}
        \centering
        \includegraphics[width=\linewidth, trim={2.1cm 2.1cm 2.1cm 2.1cm}, clip]{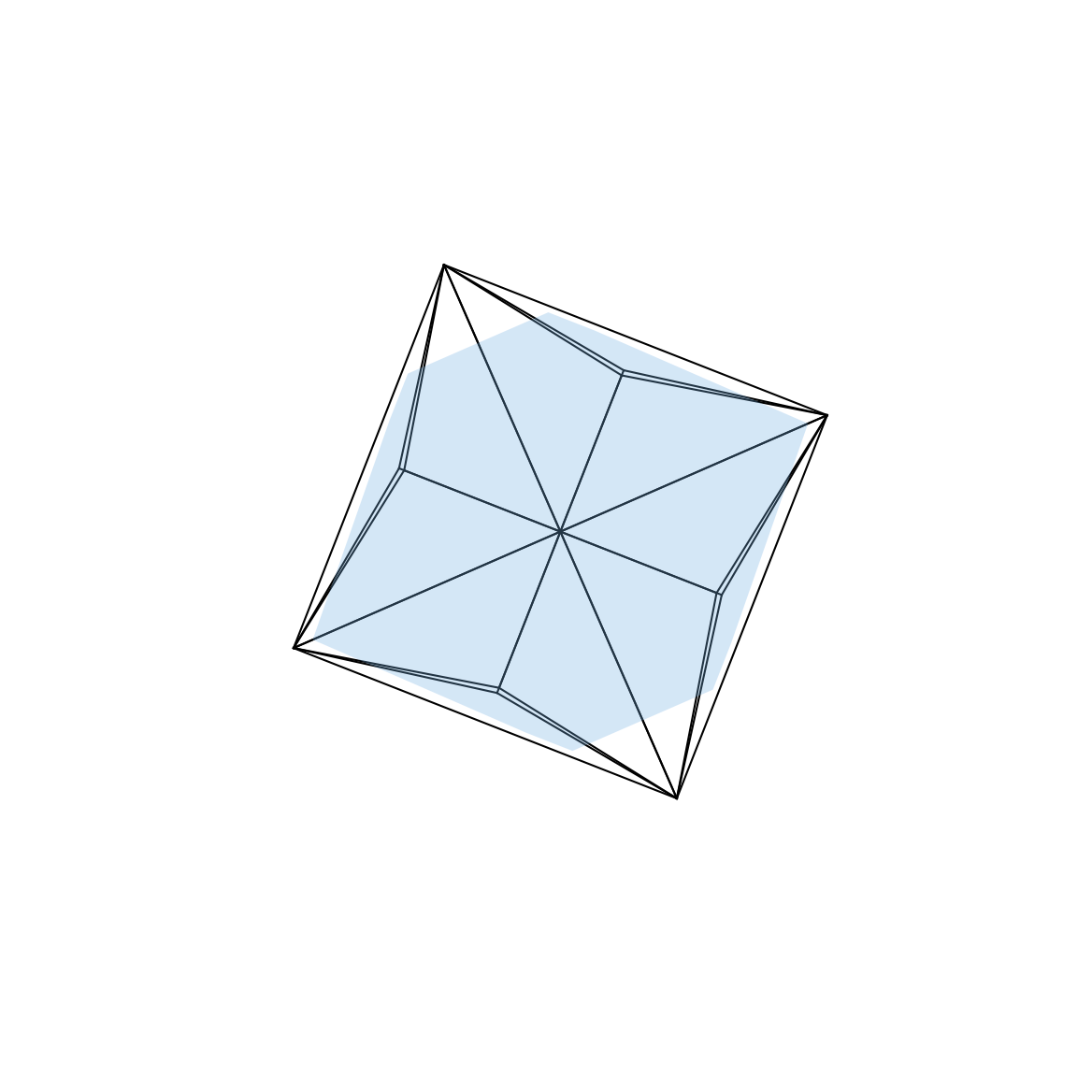}
        \caption{\centering Triakis Octahedron}
    \end{subfigure}
    \vskip\baselineskip\vskip-9pt
    \hfill
    \begin{subfigure}[b]{0.19\textwidth}
        \centering
        \includegraphics[width=\linewidth, trim={2.1cm 2.1cm 2.1cm 2.1cm}, clip]{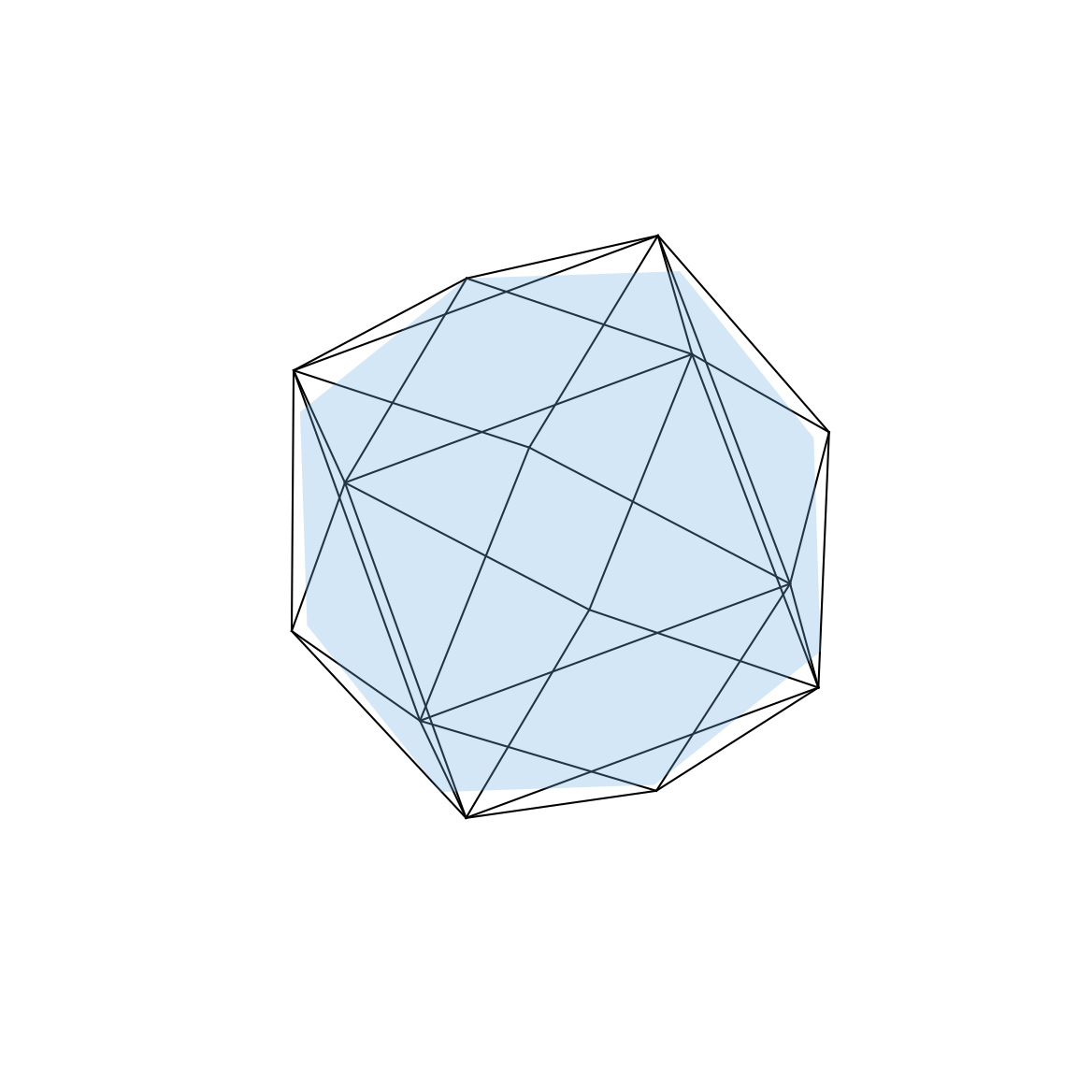}
        \caption{\centering Tetrakis Hexahedron}
    \end{subfigure}
    \hfill
    \begin{subfigure}[b]{0.19\textwidth}
        \centering
        \includegraphics[width=\linewidth, trim={2.1cm 2.1cm 2.1cm 2.1cm}, clip]{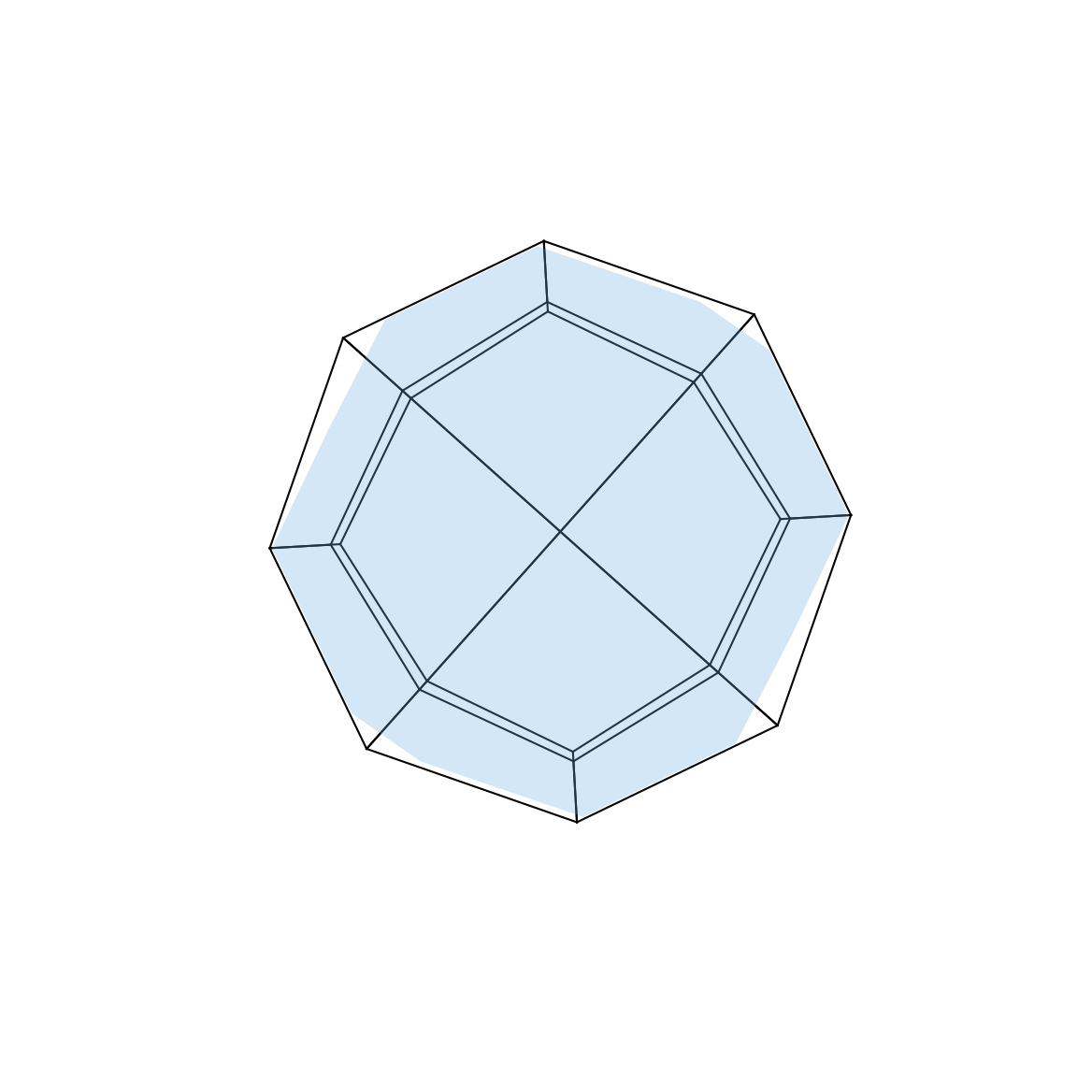}
        \caption{\centering Deltoidal Icositetrahedron}
    \end{subfigure}
    \hfill \ 
    \vskip\baselineskip\vskip-9pt
    \begin{subfigure}[b]{0.19\textwidth}
        \centering
        \includegraphics[width=\linewidth, trim={2.1cm 2.1cm 2.1cm 2.1cm}, clip]{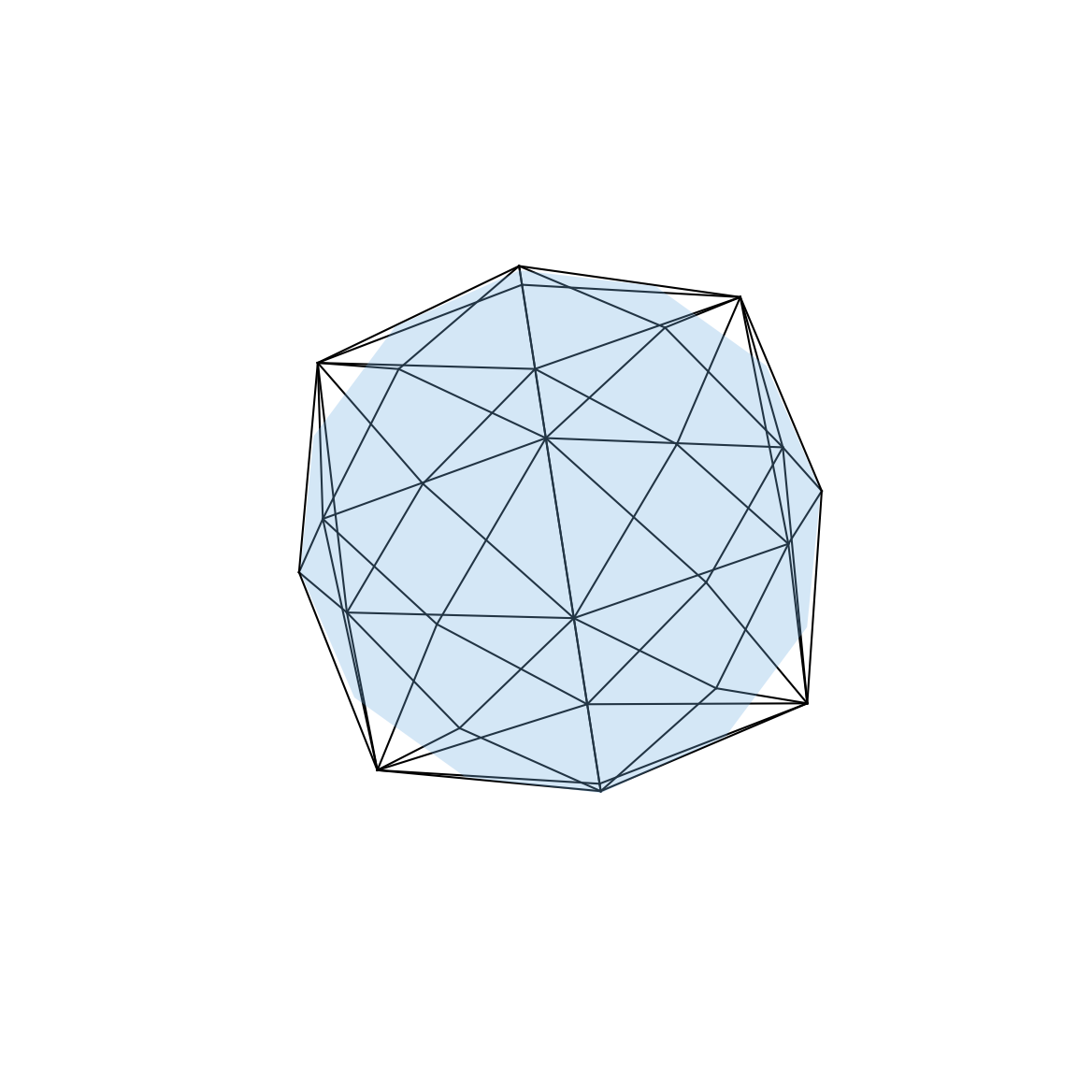}
        \caption{\centering Disdyakis Dodecahedron}
    \end{subfigure}
    \hfill
    \begin{subfigure}[b]{0.19\textwidth}
        \centering
        \includegraphics[width=\linewidth, trim={2.1cm 2.1cm 2.1cm 2.1cm}, clip]{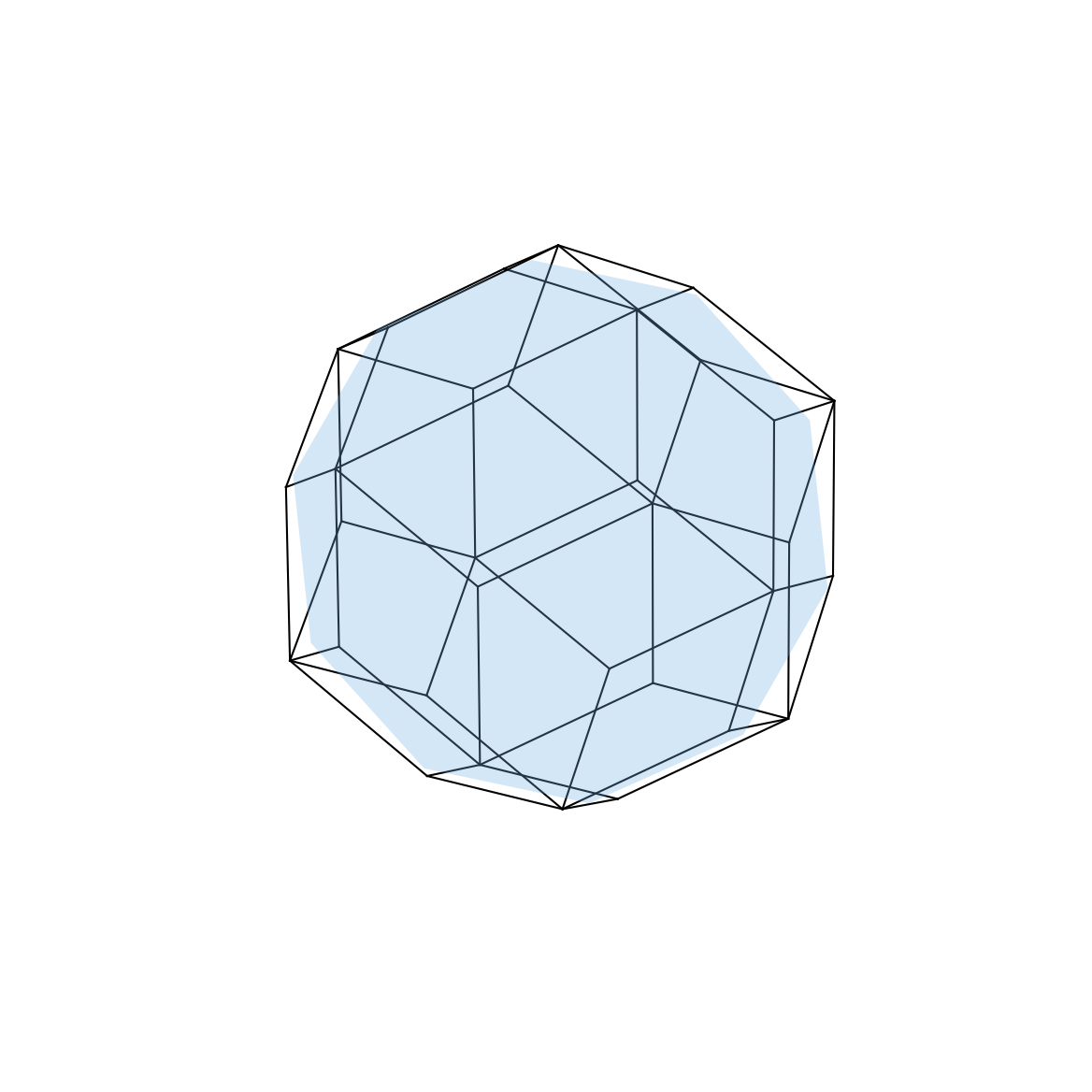}
        \caption{\centering Rhombic Triacontahedron}
    \end{subfigure}
    \hfill
    \begin{subfigure}[b]{0.19\textwidth}
        \centering
        \includegraphics[width=\linewidth, trim={2.1cm 2.1cm 2.1cm 2.1cm}, clip]{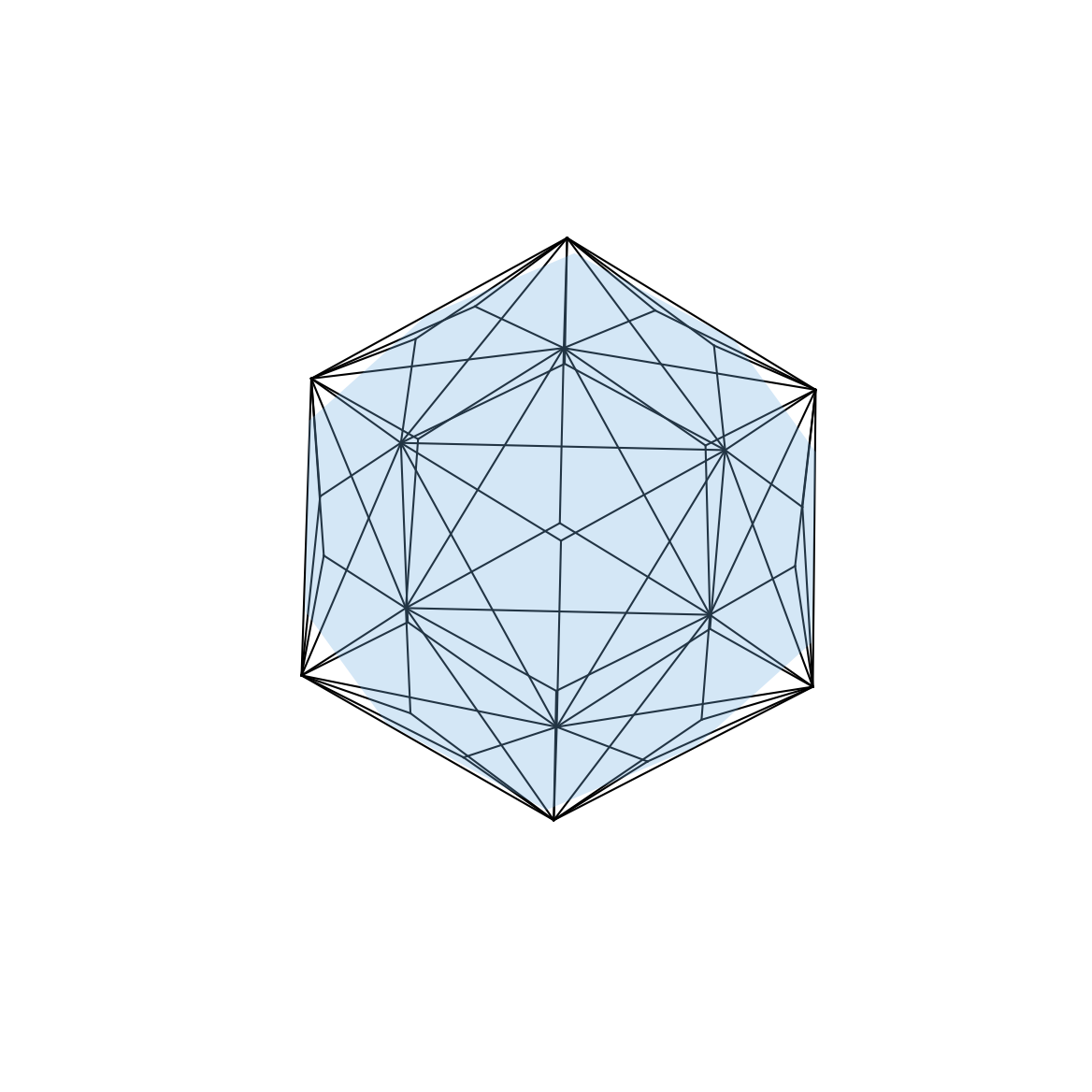}
        \caption{\centering Triakis Icosahedron}
    \end{subfigure}
    \vskip\baselineskip\vskip-9pt
    \hfill 
    \begin{subfigure}[b]{0.19\textwidth}
        \centering
        \includegraphics[width=\linewidth, trim={2.1cm 2.1cm 2.1cm 2.1cm}, clip]{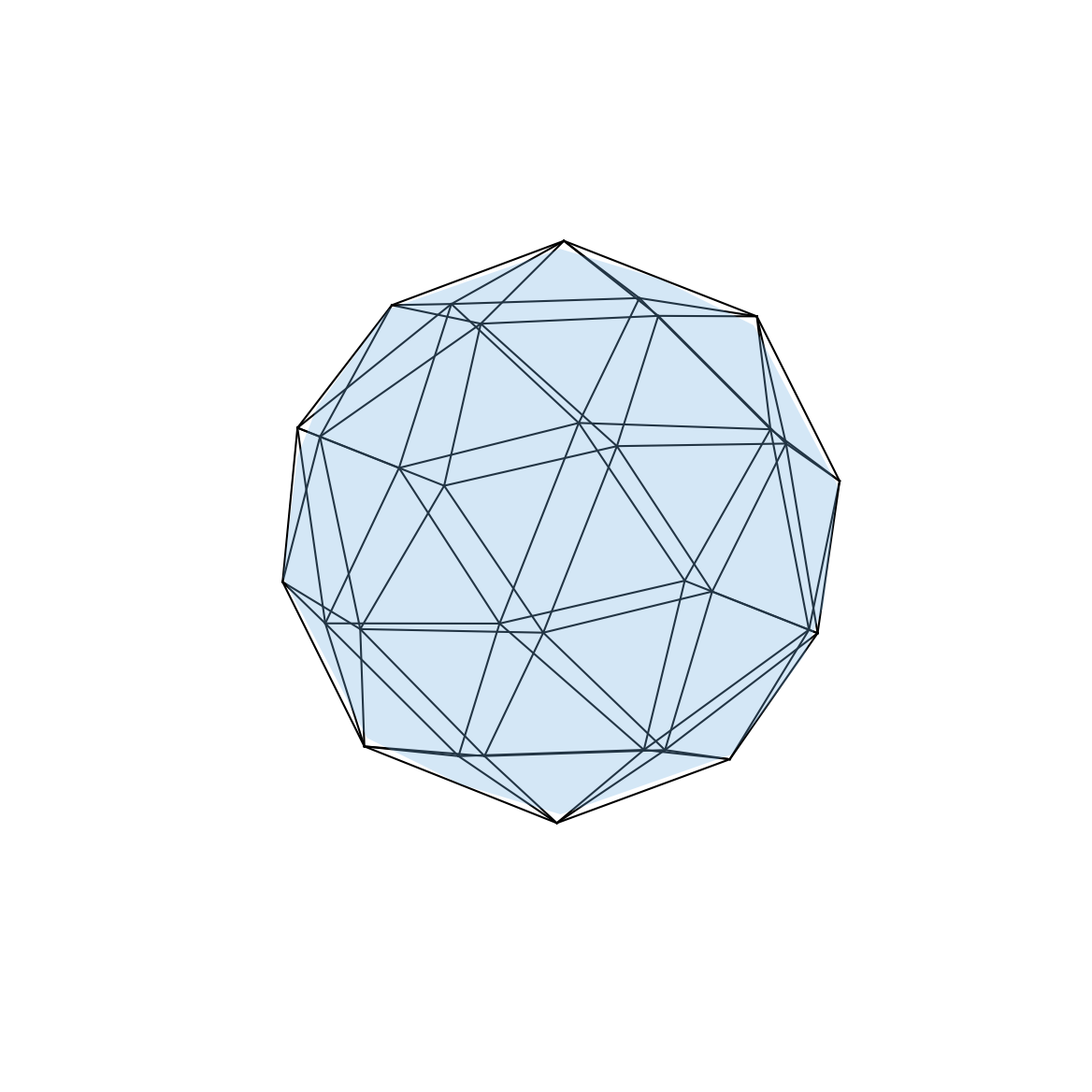}
        \caption{\centering Pentakis Dodecahedron}
    \end{subfigure}
    \hfill
    \begin{subfigure}[b]{0.19\textwidth}
        \centering
        \includegraphics[width=\linewidth, trim={2.1cm 2.1cm 2.1cm 2.1cm}, clip]{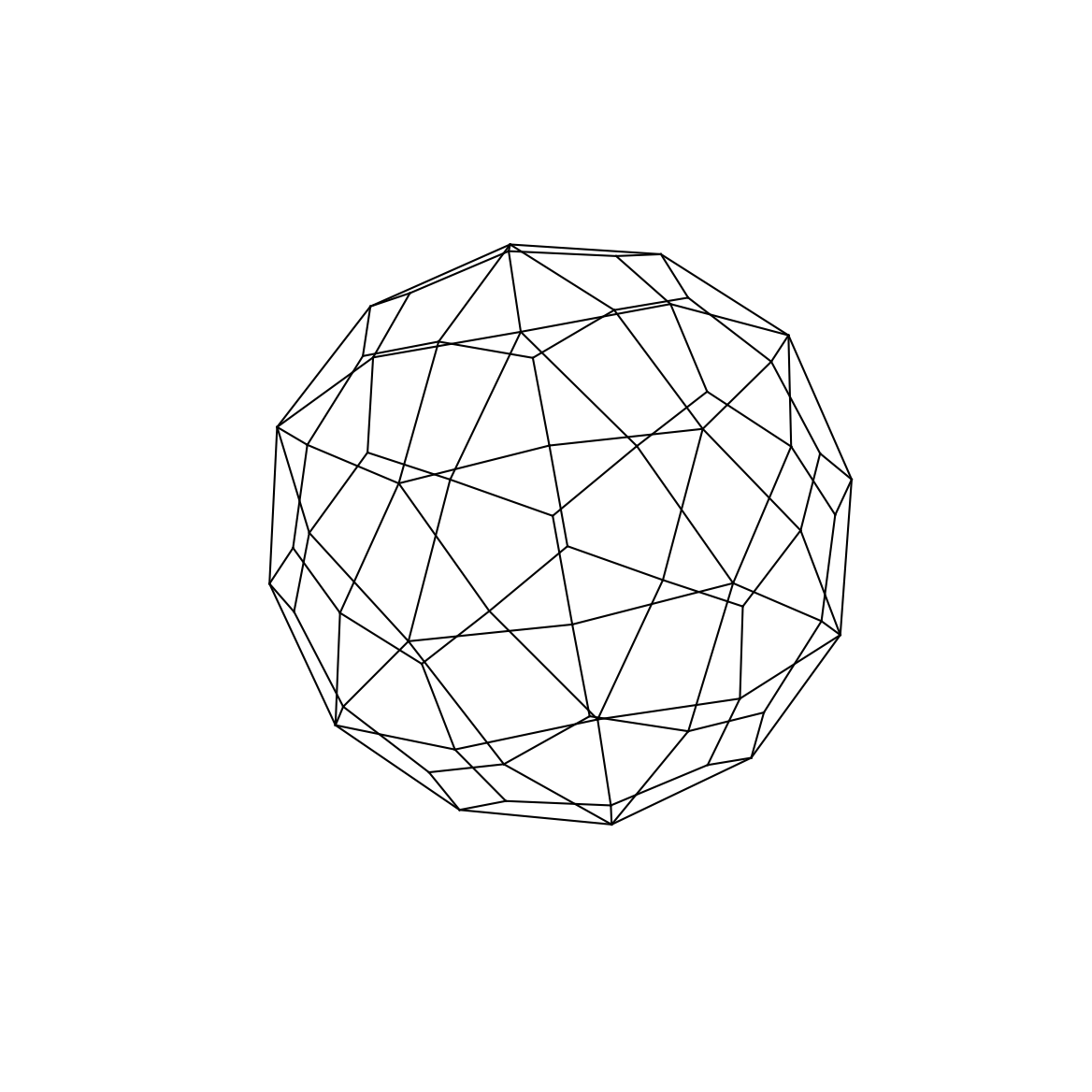}
        \caption{\centering Deltoidal Hexecontahedron}
    \end{subfigure}
    \hfill \ 
    \vskip\baselineskip\vskip-9pt
    \begin{subfigure}[b]{0.19\textwidth}
        \centering
        \includegraphics[width=\linewidth, trim={2.1cm 2.1cm 2.1cm 2.1cm}, clip]{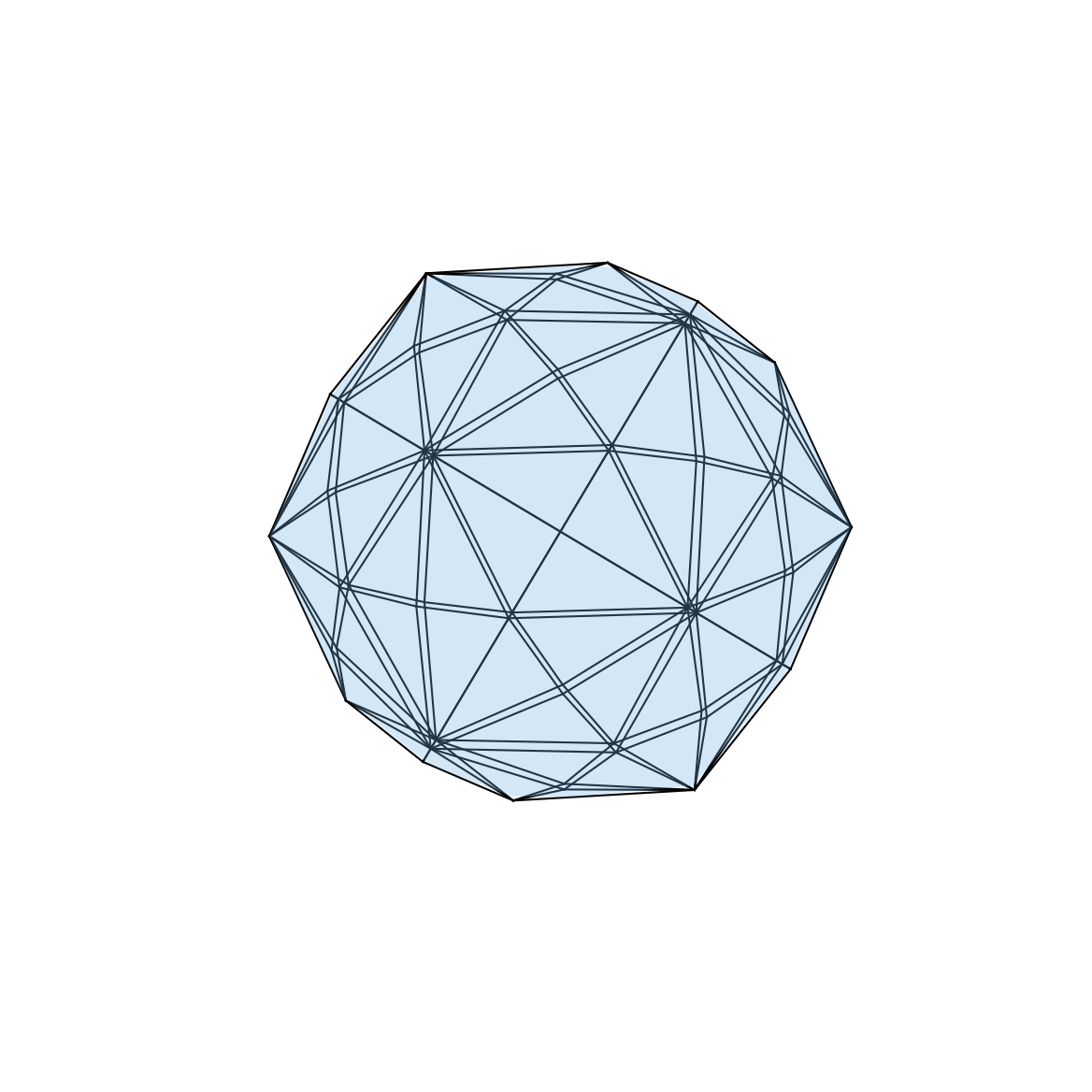}
        \caption{\centering Disdyakis Triacontahedron}
    \end{subfigure}
    \hfill
    \begin{subfigure}[b]{0.19\textwidth}
        \centering
        \includegraphics[width=\linewidth, trim={2.1cm 2.1cm 2.1cm 2.1cm}, clip]{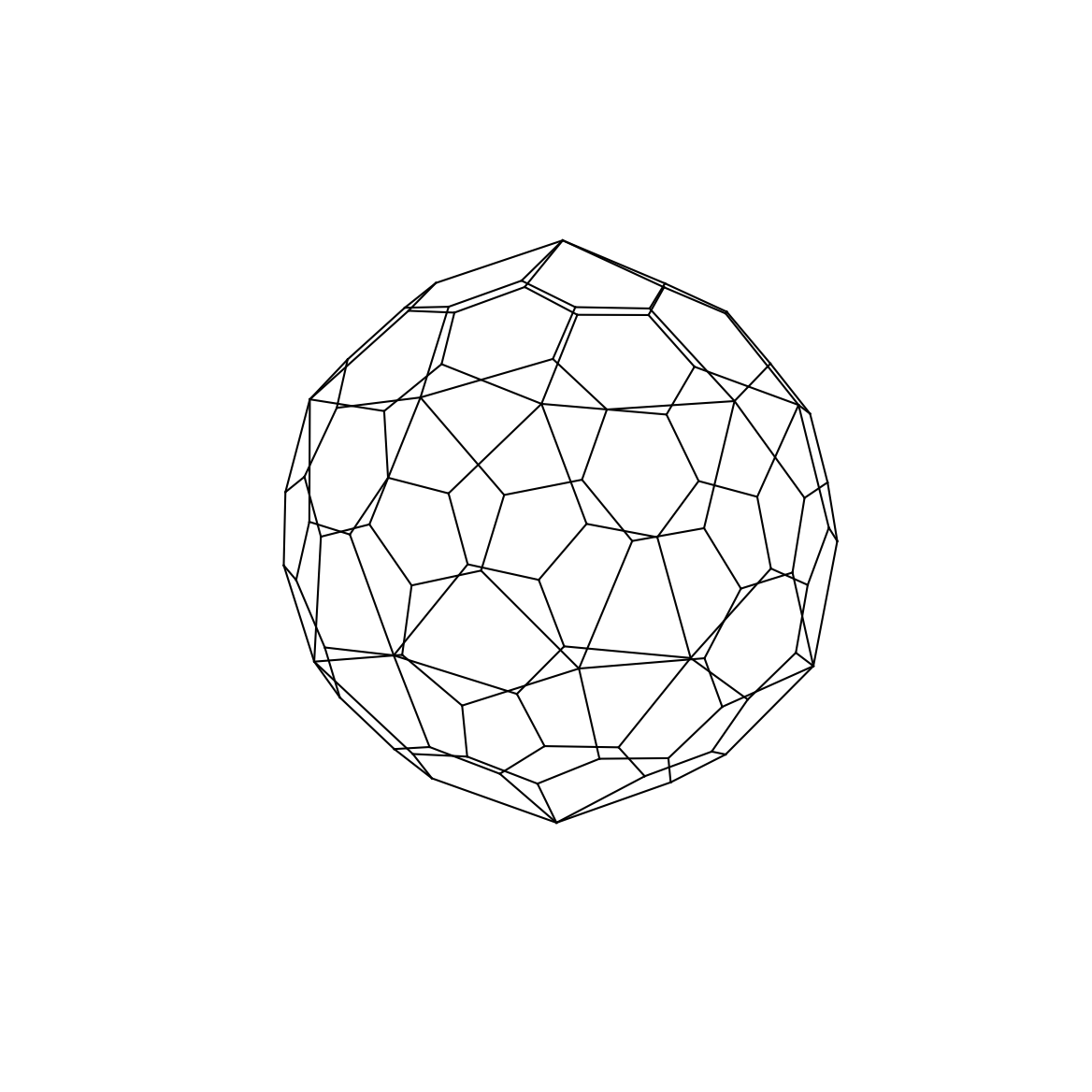}
        \caption{\centering Pentagonal Hexecontahedron}
    \end{subfigure}
    \hfill
    \begin{subfigure}[b]{0.19\textwidth}
        \centering
        \includegraphics[width=\linewidth, trim={2.1cm 2.1cm 2.1cm 2.1cm}, clip]{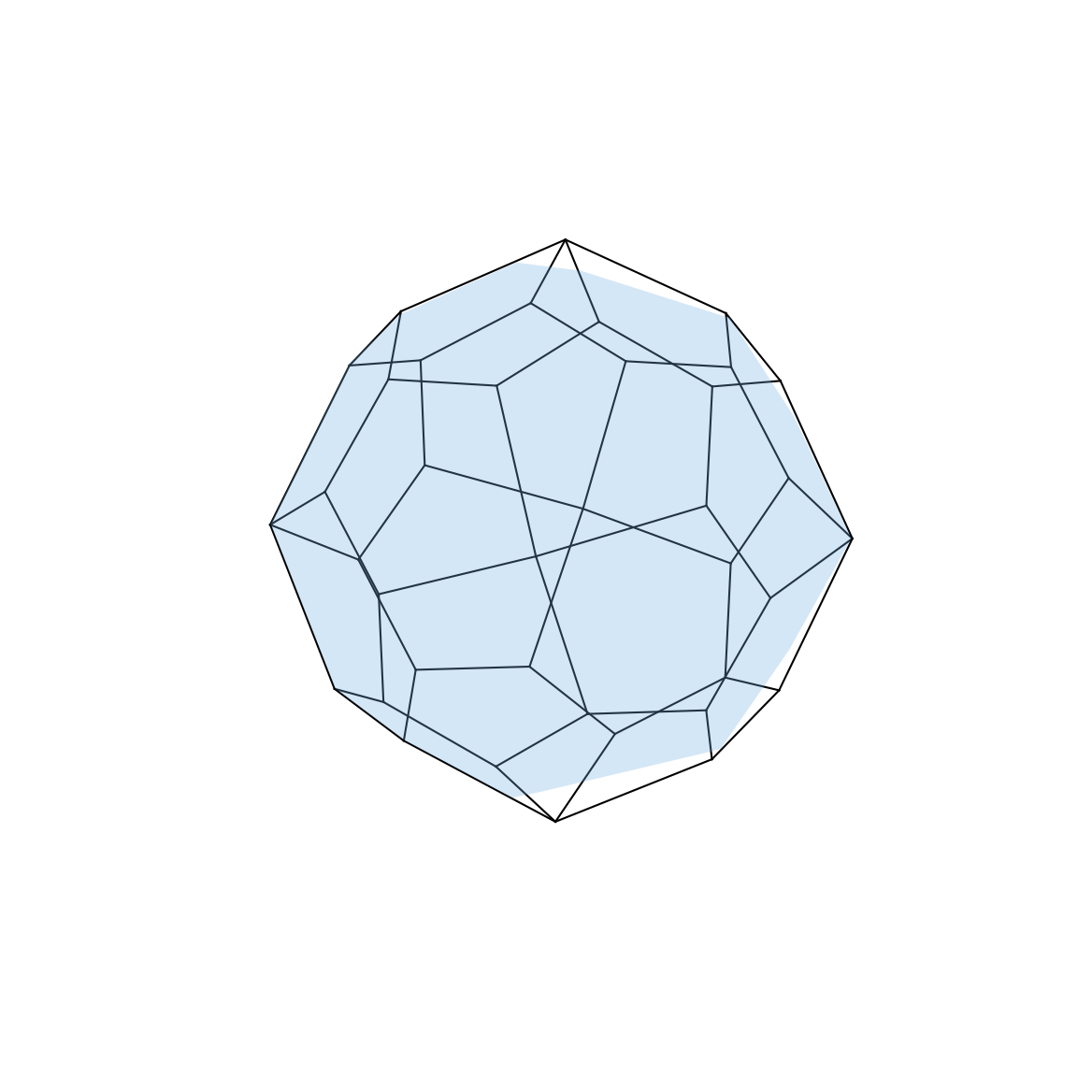}
        \caption{\centering Pentagonal Icositetrahedron}
    \end{subfigure}
   \caption{(New) Best-known, highly optimized passages for the thirteen Catalan solids and bounds on their Nieuwland constants. See Table~\ref{tab:catalan} for exact numerically computed constants.}
    \label{fig:catalan}
\end{figure}

Recall that the numerical search of Fredriksson in~\cite{fredriksson2024} showed that all but two Catalan solids and five Johnson solids are Rupert. Our local search improved many of the best-known passages among these but did not resolve any of these seven remaining shapes with open Rupertness. 

For every Catalan solid, our method matched or improved the previously best-known Nieuwland constant lower bound: For 8 of the 13 Catalan solids, our method strictly improved on the best-known passage. Further, it recovered the passages for triakis tetrahedron and pentagonal icositetrahedron, which Fredriksson~\cite{fredriksson2024} recently discovered, providing them with further digits of accuracy.

Although prior works identified passages for 87 out of the 92 Johnson solids, they did not report the best $\mu$ found. Such values only exist in the literature for the latest 5 Johnson solids that~\cite{fredriksson2024} discovered a passage for. Our method nontrivially improved the Nieuwland lower bound for four of these five: the gyroelongated pentagonal rotunda (J25), the gyroelongated square bicupola (J45), the gyroelongated pentagonal cupolarotunda (J47), and the triaugmented truncated dodecahedron (J71). Our high precision locally optimized lower bounds are presented in Table~\ref{tab:johnson}.

As a final observation, we note Conjecture~\ref{conj:dualConstants} does not appear to generalize well beyond Platonic solids. While our computations of constants for Platonic solids all agreed with their dual to at least 12 digits, substantial gaps exist between these in our high precision solves for Archimedean-Catalan pairs and Johnson duals. For example, our highly optimized passage for the truncated tetrahedron has $\mu$ exceeded one by thousands of times more than the best passage found on its dual, the triakis tetrahedron.

\begin{table}[t]
\centering
\begin{tabular}{|c|c|c|c|}
\hline
Catalan Solid & Hours & Prior Best $\mu$ & Best $\mu$ Seen \\
\hline
Triakis Tetrahedron & 48 & 1.000004\cite{fredriksson2024} & \it{1.000004055715} \\
Rhombic Dodecahedron & 48 & 1.027201 \cite{steininger2021} & \it{1.027180977829} \\
Triakis Octahedron & 48 & 1.030648\cite{steininger2021} & \bf{1.030662492707} \\
Tetrakis Hexahedron & 48 & 1.009632\cite{steininger2021} & \bf{1.011298388772} \\
Deltoidal Icositetrahedron & 48 & 1.007632\cite{steininger2021} & \bf{1.007636465375} \\
Disdyakis Dodecahedron & 48 & 1.0025\cite{steininger2021} & \bf{1.004500281028} \\
Rhombic Triacontahedron & 48 & 1.007037\cite{steininger2021} & \bf{1.007051479075} \\
Triakis Icosahedron & 48 & 1.001304\cite{steininger2021} & \bf{1.001308981561} \\
Pentakis Dodecahedron & 48 & 1.001845\cite{steininger2021} & \bf{1.001867227741} \\
Deltoidal Hexecontahedron & 180 & --- & 0.999999999999 \\
Disdyakis Triacontahedron & 48 & 1.00021\cite{steininger2021} & \bf{1.000743186023} \\
Pentagonal Hexecontahedron & 180 & --- & 0.999999999999 \\
Pentagonal Icositetrahedron & 48 & 1.000436\cite{fredriksson2024} & \it{1.000436141334} \\
\hline
\end{tabular}
\caption{Results of repeated trials optimizing passages for each Catalan solid via~\eqref{eq:subgrad} for at least 48 hours with improvements in bold and accuracy improvements in italics. Ordered such that each solid is dual to the corresponding solid in Table~\ref{tab:archimedean}.}
\label{tab:catalan}
\end{table}

\begin{table}[ht]\centering\begin{tabular}{|c|c|c|c|c|c|}\hline
 & Best $\mu$ Seen &  & Best $\mu$ Seen &  & Best $\mu$ Seen \\
\hline
J1 & \textbf{1.171572874948} & J32 & \textbf{1.022600150930} & J63 & \textbf{1.012931518383} \\
J2 & \textbf{1.051462224237} & J33 & \textbf{1.022600150839} & J64 & \textbf{1.023502155400} \\
J3 & \textbf{1.062917221135} & J34 & \textbf{1.000694527524} & J65 & \textbf{1.069918948973} \\
J4 & \textbf{1.082392200292} & J35 & \textbf{1.114806469953} & J66 & \textbf{1.037632799011} \\
J5 & \textbf{1.051462224238} & J36 & \textbf{1.115511372455} & J67 & \textbf{1.042517601422} \\
J6 & \textbf{1.022600150935} & J37 & \textbf{1.013208734428} & J68 & \textbf{1.001614361787} \\
J7 & \textbf{1.116804290106} & J38 & \textbf{1.042016668227} & J69 & \textbf{1.014484897242} \\
J8 & \textbf{1.140943757362} & J39 & \textbf{1.041930944390} & J70 & \textbf{1.003759497596} \\
J9 & \textbf{1.013466016137} & J40 & \textbf{1.007239712259} & J71 & \textbf{1.000882983008} \\
J10 & \textbf{1.010317644657} & J41 & \textbf{1.007206668263} & J72 & --- \\
J11 & \textbf{1.014643891235} & J42 & \textbf{1.046011088608} & J73 & --- \\
J12 & \textbf{1.039963426276} & J43 & \textbf{1.046022855131} & J74 & --- \\
J13 & \textbf{1.043612947553} & J44 & \textbf{1.006295254109} & J75 & --- \\
J14 & \textbf{1.116804339142} & J45 & \textbf{1.000020986354} & J76 & \textbf{1.000228660525} \\
J15 & \textbf{1.224744871266} & J46 & \textbf{1.000226849716} & J77 & --- \\
J16 & \textbf{1.035700492441} & J47 & \textbf{1.000082754534} & J78 & \textbf{1.001845783918} \\
J17 & \textbf{1.012458022219} & J48 & \textbf{1.002390043064} & J79 & \textbf{1.000397518135} \\
J18 & \textbf{1.039869879356} & J49 & \textbf{1.069686039279} & J80 & \textbf{1.002288221901} \\
J19 & \textbf{1.039736400938} & J50 & \textbf{1.124651673789} & J81 & \textbf{1.002764894770} \\
J20 & \textbf{1.042016605980} & J51 & \textbf{1.008115252776} & J82 & \textbf{1.003010991091} \\
J21 & \textbf{1.007297585982} & J52 & \textbf{1.036449361724} & J83 & \textbf{1.002269779677} \\
J22 & \textbf{1.001078125630} & J53 & \textbf{1.098120497031} & J84 & \textbf{1.168675494427} \\
J23 & \textbf{1.000303695345} & J54 & \textbf{1.099668218694} & J85 & \textbf{1.029541060803} \\
J24 & \textbf{1.000434587899} & J55 & \textbf{1.151612677378} & J86 & \textbf{1.080318705772} \\
J25 & \textbf{1.000082039977} & J56 & \textbf{1.125778230778} & J87 & \textbf{1.062410334923} \\
J26 & \textbf{1.043994054861} & J57 & \textbf{1.019956076347} & J88 & \textbf{1.261027884055} \\
J27 & \textbf{1.044484298401} & J58 & \textbf{1.014157239071} & J89 & \textbf{1.029137337968} \\
J28 & \textbf{1.082392200292} & J59 & \textbf{1.028772884206} & J90 & \textbf{1.019733439708} \\
J29 & \textbf{1.082392200292} & J60 & \textbf{1.014961907470} & J91 & \textbf{1.203409213649} \\
J30 & \textbf{1.051462224238} & J61 & \textbf{1.006280581240} & J92 & \textbf{1.070466268215} \\
J31 & \textbf{1.051462224238} & J62 & \textbf{1.023269308745} &  & \\

\hline\end{tabular}
\caption{Results of repeated trials optimizing passages for each Johnson solid via~\eqref{eq:subgrad} for at least 16 hours each with improvements over prior best bounds denoted in bold. Note for most Johnson solids, the only bounds reported in prior works we can compare with is that a passage was found with $\mu$ strictly larger than one (without providing an explicit numerical value).}
\label{tab:johnson}
\end{table}

\section{Conclusion.}
Classic convex polyhedra with unknown Rupertness remain. Our methods provided nontrivial improvements in the constants of more than half of the Platonic, Archimedean, and Catalan solids and most of the Johnson solids. These techniques have likely exhausted the level of improvements possible from growing precision and computational budgets in iterative local optimization. As a result, future success in proving or disproving the Rupertness of these remaining polyhedra may require the development of novel methods. Our Conjecture~\ref{conj:tetrahedron} presents a structured target for future work, establishing whether our discovered $\mu = \frac{\sqrt{6}}{1+\sqrt{2}}$ tetrahedron passage is optimal. 
        
    {\small
    \bibliographystyle{unsrt}
    \bibliography{MonthlyReferences}
    }
\end{document}